\documentclass{amsart}

\usepackage{amsfonts,amssymb,amsmath,enumerate,verbatim,mathtools,tikz,bm,mathrsfs,tikz-cd,hyperref,comment,stmaryrd,amsthm}

\usepackage{colonequals}
\usepackage{adjustbox}
\usepackage{leftindex}
\usepackage[all,pdf]{xy}
\hypersetup{
    colorlinks=true,
    linkcolor=blue,
    filecolor=magenta,      
    urlcolor=cyan,
}

\usepackage{thmtools}
\usepackage{cleveref}

\makeatletter
\@namedef{subjclassname@2020}{\textup{2020} Mathematics Subject Classification}
\makeatother

\author[Souvik Dey]{Souvik Dey}
\address{Department of Mathematical Sciences, 850 West Dickson Street, University of Arkansas, Fayetteville, Arkansas 72701, USA, https://orcid.org/0000-0001-8265-3301}
\email{souvikd@uark.edu}

\author[Jian Liu]{Jian Liu}
\address{School of Mathematics and Statistics, and Hubei Key Laboratory of Mathematical Sciences,  Central China Normal University,  Wuhan 430079, P.R. China,\newline https://orcid.org/0000-0001-8360-7024}
\email{jianliu@ccnu.edu.cn}

\author[Xue-Song Lu]{Xue-Song Lu}
\address{School of Mathematical Sciences, Shanghai Jiao Tong University, Shanghai 200240, P.R. China, https://orcid.org/0009-0001-4439-4684}
\email{leocedar@sjtu.edu.cn}


\renewcommand{\S}{{\mathcal{S}}}

\DeclareMathOperator{\Ab}{Ab}

\DeclareMathOperator{\depth}{depth}

\DeclareMathOperator{\res}{res}
\DeclareMathOperator{\cx}{cx}
\DeclareMathOperator{\h}{H}

\newcommand{\syz}{{\Omega}}

\newcommand{\X}{\mathcal{X}}

\newcommand{\Z}{\mathbb{Z}}

\newcommand{\A}{\mathcal{A}}

\newcommand{\K}{\mathsf{K}}

\pgfdeclarelayer{bg}    
\pgfsetlayers{bg,main} 
\newcommand{\x}{{\bm{x}}}

\newcommand{\del}{\partial}

\newcommand{\m}{\mathfrak{m}}

\newcommand{\p}{\mathfrak{p}}

\DeclareMathOperator{\Ima}{Im}

\DeclareMathOperator{\Coker}{Coker}

\DeclareMathOperator{\add}{add}

\DeclareMathOperator{\id}{id}

\DeclareMathOperator{\Hom}{Hom}

\DeclareMathOperator{\Ext}{Ext}

\DeclareMathOperator{\End}{End}

\DeclareMathOperator{\Thick}{\mathsf{Thick}}
\DeclareMathOperator{\Tor}{Tor}

\DeclareMathOperator{\MCM}{MCM}

\DeclareMathOperator{\RHom}{\mathsf{RHom}}

\newcommand{\Mod}{\mathsf{Mod}}
\newcommand{\Inj}{\mathsf{Inj}}
\newcommand{\mo}{\mathsf{mod}}

\newcommand{\proj}{\mathsf{proj}}
\newcommand{\inj}{\mathsf{inj}}

\newcommand{\Proj}{\mathsf{Proj}}

\newcommand{\GProj}{\mathsf{GProj}}
\newcommand{\GInj}{\mathsf{GInj}}
\newcommand{\Gproj}{\mathsf{Gproj}}

\newtheorem{theorem}{Theorem}[section]

\newtheorem{proposition}[theorem]{Proposition}

\newtheorem{lemma}[theorem]{Lemma}
\newtheorem{corollary}[theorem]{Corollary}

\theoremstyle{definition}
\newtheorem{example}[theorem]{Example}
\newtheorem{remark}[theorem]{Remark}

\newtheorem{chunk}[theorem]{}

\usepackage[utf8]{inputenc}

\newtheorem*{ack}{Acknowledgements}

\title[Cotorsion pairs, thick subcategories, and Gorenstein projectives]{
 Cotorsion pairs, thick subcategories, and finitely generated Gorenstein projective modules}
\date{\today}

\begin{document}

\begin{abstract}
Let $R$ be a noetherian algebra over a Cohen--Macaulay ring $S$ admitting a canonical module $\omega$, and assume that $R$ is maximal Cohen--Macaulay over $S$. We prove that the category of finitely generated Gorenstein projective $R$-modules coincides with the left $\Ext$-orthogonal class of the thick subcategory generated by $R$ and $\Hom_S(R,\omega)$. As an application, finitely generated Gorenstein projective $R$-modules form the left half of a hereditary cotorsion pair. In the case of Cohen--Macaulay local rings, this yields an affirmative answer to a question of R. Takahashi. We further characterize when $R$ is left weakly Gorenstein.  Finally, we prove that a Cohen--Macaulay local ring is Gorenstein if and only if the right $\Ext$-orthogonal class of finitely generated Gorenstein projective modules coincides with the category of finitely generated modules of finite projective dimension.
\end{abstract}

\maketitle

\section{Introduction}
The study of Gorenstein homological properties has played a central role in modern representation theory and commutative algebra.  In 1967, M. Auslander \cite{A} introduced the notion of finitely generated Gorenstein projective modules in his four lectures at the S\'eminaire Pierre Samuel, under the name ``modules of {\rm G}-dimension zero''. 
In \cite{Enochs-Jenda-MZ}, E. E. Enochs and O. M. G. Jenda studied Gorenstein projective and injective modules for arbitrary rings. For an Iwanaga--Gorenstein ring, R.-O. Buchweitz \cite{Buchweitz} established a triangle equivalence between the stable category of finitely generated Gorenstein projective modules and the singularity category, demonstrating deep connections between the Gorenstein projective modules and the singularity theory. 

\vskip5pt

Cotorsion pairs offer a powerful framework for understanding approximation theory and resolving subcategories in module categories. A recent result, observed by M. Cortés-Izurdiaga and J. Šaroch \cite{CIS}, states that $(\GProj(R), \GProj(R)^{\perp_1})$ forms a cotorsion pair in the category of $R$-modules, where $\GProj(R)$ denotes the category of all Gorenstein projective modules and $\GProj(R)^{\perp_1}$ is the right $\Ext^1$-orthogonal class of $\GProj(R)$. The case of Artin algebras was previously established by A. Beligiannis and I. Reiten \cite{Beligiannis-Reiten}. However, the situation in the finitely generated setting remains unclear in general. Our first result, \Cref{T2}, shows that the category $\Gproj(R)$ of finitely generated Gorenstein projective modules occurs as the left half of a cotorsion pair for certain noetherian algebras over Cohen--Macaulay rings. 

\vskip5pt

\begin{theorem}\label{T2} (See \ref{cotorsion pair} and \ref{takanswer1})
 Let $S$ be a Cohen--Macaulay ring. Assume that $R$ is a noetherian $S$-algebra and that $R$, viewed as an $S$-module, is maximal Cohen--Macaulay. If moreover $S$ is local or admits a canonical module, then there is a hereditary cotorsion pair $(\Gproj(R),\Gproj(R)^{\perp_1})$ in $\mo(R)$.
\end{theorem}
\Cref{T2} gives an affirmative answer to a question of R. Takahashi \cite{Takahashi} for Cohen--Macaulay local rings; see \Cref{Support}. Its proof relies in part on a characterization of $\Gproj(R)$ established in \Cref{characterization}, namely, as the left $\Ext^1$-orthogonal class of the thick subcategory $\Thick(\proj(R)\cup \{\Hom_S(R,\omega)\})$ in the category $\mo(R)$ of finitely generated modules.  When $R$ is an Artin algebra over a commutative Artin ring $S$, this thick subcategory coincides with $ \Thick(\proj(R)\cup \inj(R))$. This thick category was first systematically studied by A. Beligiannis and H. Krause \cite{BK}. They proved that an Artin algebra $R$ is virtually Gorenstein if and only if $\Thick(\proj(R)\cup \inj(R))$ is contravariantly finite in $\mo(R)$, if and only if $\Thick(\proj(R)\cup \inj(R))$ is covariantly finite in $\mo(R)$. 

\vskip5pt

Since $\Gproj(R)$ is a resolving subcategory, it is natural to ask whether every resolving subcategory induces a cotorsion pair. In general, however, this is not the case. Indeed, we construct examples of resolving subcategories $\mathcal X$ of $\mo(R)$ for which the pair $(\mathcal X,\mathcal X^{\perp_1})$ is not a cotorsion pair; see \Cref{counterexample}.

\vskip5pt

The notion of \emph{left weakly Gorenstein rings} was introduced by C. M. Ringel and P. Zhang \cite{RZ}. It was also independently introduced under different names by L. W. Christensen and H. Holm \cite{CH}, as well as by R. Marczinzik \cite{Rene2021}. By definition, a left weakly Gorenstein ring is a noetherian ring for which the left Ext-orthogonal class of the ring itself coincides with the class of finitely generated Gorenstein projective left modules. In \cite{RZ}, several equivalent characterizations of left weakly Gorensteinness for Artin algebras were established. 
Recently, weakly Gorensteinness has garnered increasing interest and has been studied in \cite{GZZ,Huang,Rene2019}.

\vskip5pt

By combining \Cref{T2} with the characterization of monic left $\proj(R)$-approximations established in \Cref{cm case-approximation}, we obtain the second main result \Cref{T1} of this article. This theorem characterizes left weakly Gorenstein rings in terms of the thick subcategory
$
\Thick(\proj(R)\cup {\Hom_S(R,\omega)}).
$
It further shows that the left weakly Gorenstein property is determined by a single module.
It is worth noting that every Artin algebra and every complete Cohen--Macaulay local ring satisfies the hypotheses of \Cref{T1}; see \Cref{Examples} for further examples.

\begin{theorem}\label{T1}(See \ref{char-weakly Gor})
   Let $S$ be a Cohen--Macaulay ring that admits a canonical module $\omega$. Assume that $R$ is a noetherian $S$-algebra and that $R$, viewed as an $S$-module, is maximal Cohen--Macaulay.  Then the following are equivalent.
\begin{enumerate}
    \item $R$ is left weakly Gorenstein. 

    \item $\Thick(\proj(R)\cup \{\Hom_S(R,\omega)\})\subseteq (\leftindex^{\perp_\infty} {R})^{\perp_\infty}$. 

    \item $X\in (\leftindex^{\perp_\infty} R)^{\perp_\infty}$, where $X$ is a first syzygy of the $R$-module $\Hom_S(R,\omega)$.

\end{enumerate}
\end{theorem}

The category $(\leftindex^{\perp_\infty} {R})^{\perp_\infty}$ above is the double $\Ext$-orthogonal class of $R$ in the category of finitely generated $R$-modules; see \Cref{def-cotorsion}.

\vskip5pt

The notion of \emph{virtually Gorenstein algebra} was introduced by A. Beligiannis and I. Reiten \cite{Beligiannis-Reiten}. This notion was extended to commutative noetherian rings of finite Krull dimension by F. Zareh-Khoshchehreh, M. Asgharzadeh, and K. Divaani-Aazar \cite{ZAD}. More recently, Z. Di, L. Liang, and J. Wang \cite{DLW} generalized the concept further to arbitrary rings. 
Let $\Lambda$ be an Artin algebra such that $\Gproj(\Lambda) = \proj(\Lambda)$. In \cite[Problem C]{Chen-postdoc}, X.-W. Chen asked whether $\GProj(\Lambda) = \Proj(\Lambda)$. \Cref{T4} shows that this problem is equivalent to asking whether $\Lambda$ is virtually Gorenstein.

\begin{theorem}\label{T4} (See \ref{char-VG})
    Let $\Lambda$ be an Artin algebra and assume $\Gproj(\Lambda)=\proj(\Lambda)$. The following are equivalent. 
    \begin{enumerate}
        \item $\Lambda$ is virtually Gorenstein.

        \item $\GProj(\Lambda)=\Proj(\Lambda)$.

        \item $\Thick(\proj(\Lambda)\cup \inj(\Lambda))=\mo(\Lambda)$. 

        \item $
    (\leftindex^{\perp_1}\Thick(\proj(\Lambda)\cup \inj(\Lambda)), \Thick(\proj(\Lambda)\cup \inj(\Lambda)))
    $
    is a cotorsion pair.
    \end{enumerate} 
\end{theorem}

There is a classical result that $(\Gproj(R), {\mathcal P}^{<\infty}(R))$ forms a cotorsion pair when $R$ is Iwanaga--Gorenstein; here, ${\mathcal P}^{<\infty}(R)$ denotes the full subcategory of $\mo(R)$ consisting of all finitely generated
 modules of finite projective dimension. A natural question is whether the converse also holds. The following result shows that this is indeed the case for Cohen--Macaulay local rings. This provides new insight into the homological structure of Iwanaga--Gorenstein rings via the cotorsion pair generated by finitely generated Gorenstein projective modules.

\begin{theorem}\label{T3} (See \ref{result3})
    Let $R$ be a Cohen--Macaulay local ring. The following are equivalent. 
    \begin{enumerate}
        \item $R$ is Gorenstein.

        \item $\Gproj(R)^{\perp_1}={\mathcal P}^{<\infty}(R)$.

        \item $(\Gproj(R), {\mathcal P}^{<\infty}(R))$ is a cotorsion pair.
    \end{enumerate}
\end{theorem}

This paper is organized as follows. We recall basic notions and classical results in Section 2. 
In Section 3, we first prove that if $S$ is a Cohen--Macaulay ring admitting a canonical module $\omega$, and $R$ is a noetherian $S$-algebra which is maximal Cohen--Macaulay as an $S$-module, then the class of finitely generated Gorenstein projective modules is exactly the left $\Ext^1$-orthogonal class of the thick subcategory generated by $R$ and $\Hom_S(R, \omega)$ in $\mo(R)$; see \Cref{characterization}. Using this, we prove \Cref{T2}; see \Cref{cotorsion pair} and \Cref{takanswer1}. We also prove a version of \Cref{takanswer1} for Tor-pairs; see \Cref{Tor-pair}. We give \Cref{counterexample}, showing that for a resolving subcategory $\mathcal X$ of $\mo(R)$, $(\mathcal X, \mathcal X^{\perp_1})$ may not be a cotorsion pair nor $(\X,\X^{\top_1})$ be a Tor-pair in general. When $R$ is an Artin algebra, the category of finitely generated Gorenstein projective modules is the left $\Ext^1$-orthogonal class of some full subcategory of countable type; see \Cref{char-GP-Artin}. As an application, we show that a semi-Gorenstein projective module of finite Auslander bound is Gorenstein projective; see \Cref{Auslanderbound}. 
We prove \Cref{T1} and \Cref{T4} in Section 4. Finally, we prove \Cref{T3} in Section 5.

\begin{ack}
 Part of this work was carried out while Jian Liu and Xue-Song Lu were attending the CIMPA School ``Perspectives in Non-commutative Algebras'', held at Anhui University in September 2025. We thank the organizers, especially Huanhuan Li, for their warm hospitality. We also thank Silu Liu for helpful comments.  Xue-Song Lu is grateful to his advisor, Professor Pu Zhang, for his crucial guidance and helpful suggestions on this work, and he also thanks Shijie Zhu for an inspiring private talk.  

 \vskip5pt

 Xue-Song Lu was supported in part by the National Key Research and Development Project 2025YFA1017202, and the National Natural Science Foundation of China (No. 12131015). Souvik Dey was partially supported by the Charles University Research Center program No.UNCE/SCI/022 and a grant GA \v{C}R 23-05148S from the Czech Science Foundation.
    Jian Liu was supported by the National Natural Science Foundation of China (No. 12401046).

\end{ack}

\section{Preliminaries}
In this article, a noetherian ring is assumed to be two-sided noetherian. 
Throughout, $R$ is always a noetherian ring. Denote by $\Mod(R)$ (resp. $\mo(R)$) the category of all (resp. finitely generated) left $R$-modules, and by $\Proj(R)$ (resp. $\Inj(R)$) the full subcategory of $\Mod(R)$ consisting of all projective (resp. injective) $R$-modules. Denote by $R^{\rm op}$ the opposite ring of $R$. Consequently, a right $R$-module can be viewed as a module in $\Mod(R^{\rm op})$. We will also write $_RM$ (resp. $N_R$) to indicate that $M$ (resp. $N$) is a left (resp. right) $R$-module.

\vskip5pt

We write $\proj(R)$ (resp. $\inj(R)$) for the full subcategory of $\Proj(R)$ (resp. $\Inj(R)$) consisting of all finitely generated projective (resp. injective) modules.  
We use ${\mathcal P}^{<\infty}(R)$ (resp. $\mathcal I^{<\infty}(R)$) to denote the full subcategory of $\mo(R)$ consisting of all finitely generated modules of finite projective (resp. injective) dimension. For an $R$-module $M$, we write $\id_R(M)$ to denote the injective dimension of $M$ over $R$.

\vskip5pt

When $R$ is commutative, we denote by $\dim(R)$ the Krull dimension of $R$.
If, moreover, $R$ is a local ring, then for each $M \in \mo(R)$, we denote by
$\depth(M)$ the depth of $M$ over $R$; that is, the length of a maximal
$M$-regular sequence contained in the maximal ideal of $R$.

\begin{chunk}\textbf{Noetherian algebras.}
     For a commutative noetherian ring $S$, $R$ is said to be a \emph{noetherian $S$-algebra} if there exists a ring homomorphism $\varphi\colon S\rightarrow R$ such that the image of $\varphi$ is in the center of $R$ and $R$ is finitely generated as an $S$-module.
\end{chunk}

\begin{chunk}\label{duality}\textbf{Artin algebras.}
    Let $\Lambda$ be an Artin algebra; that is, there exists a commutative Artin ring $S$ such that $\Lambda$ is an $S$-algebra and is finitely generated as an $S$-module.  Let $J$ be the injective envelope of $\oplus_{i=1}^nS_i$ over $S$, where $S_1,\ldots,S_n$ are all non-isomorphic simple $S$-modules. Set $D=\Hom_S(-,J)$.  It defines natural contravariant functors
\[
\begin{tikzcd}
\Mod(\Lambda)\arrow[r,shift left=0.8ex,"D"]&\arrow[l,shift left=0.8ex,"D"] \Mod(\Lambda^{\rm op})
\end{tikzcd}
\]
It is known that $D$ restricts to a duality between $\mo(\Lambda)$ and $\mo(\Lambda^{\rm op})$.
\end{chunk}

\begin{chunk}\label{def-thick}
    \textbf{Thick subcategories of an abelian category.} A full subcategory $\mathcal C$ of an abelian category $\A$ is called \emph{thick} if it is closed under direct summands and satisfies that if two out of three terms in a short exact sequence are in $\mathcal C$, then so is the third one. For a class of objects $\mathcal S$ in $\A$, we write $\Thick_\A (\mathcal S)$ to be the smallest thick subcategory of $\A$ containing $\mathcal S$. If $\S$ is a class of modules in $\mo(R)$, we will simply use $\Thick(\S)$ to represent $\Thick_{\mo(R)}(\S)$ in this article.
\end{chunk}

\begin{chunk}\label{resolving} 
\textbf{Resolving subcategories of an abelian category}
     A full subcategory $\mathcal C$ of an abelian category $\A$ is called \emph{resolving} if it contains all projective objects and is closed under extensions, direct summands, and kernels of epimorphisms. When $\A$ contains enough projectives, the closure under kernel of epimorphisms can be replaced by the weaker condition of closure under syzygies. 
\end{chunk}

\begin{chunk}\label{torsionless}
\textbf{Torsionless modules and reflexive modules.}
For each $M\in\mo(R)$, consider the map
$$
\varphi_M\colon M\rightarrow \Hom_{R^{\rm op}}(\Hom_R(M,R),R); m\mapsto (f\mapsto f(m)).
$$
A module $M$ is called \emph{torsionless} if $\varphi_M$ is a monomorphism, and \emph{reflexive} if $\varphi_M$ is an isomorphism.
It is known that $M$ is torsionless if and only if $M$ is a submodule of some projective module.
\end{chunk}

\begin{chunk}
    \textbf{Maximal Cohen--Macaulay modules and Cohen--Macaulay rings.}\label{mcm}
 Let $R$ be a commutative noetherian ring. A module $M\in\mo(R)$ is called \emph{maximal Cohen--Macaulay} if $\depth(M_\p)\geq\dim(R_\p)$ for each prime ideal $\p$ of $R$; note that the zero module is maximal Cohen--Macaulay. We refer the reader to \cite[2.1]{BH} for more details. 

\vskip5pt
 
 Let $\MCM(R)$ denote the full subcategory of $\mo(R)$ consisting of all maximal Cohen--Macaulay modules. 
The ring $R$ is said to be \emph{Cohen--Macaulay} provided that $R\in \MCM(R)$. 

\end{chunk}

\begin{chunk}\label{def of canonical}
    \textbf{Canonical modules.}
For a Cohen--Macaulay local ring $(R,\m,k)$, where $\m$ is the maximal ideal and $k$ is the residue field $R/\m$, a module $\omega\in\mo(R)$ is said to be \emph{canonical} if $\Ext^d_R(k,\omega)\cong k$ and $\Ext^i_R(k,\omega)=0$ for $i\neq d$, where $d=\dim(R)$; see \cite[3.3]{BH}. In this case, $\id_R(\omega)<\infty$. 

\vskip5pt

In general, for a Cohen--Macaulay ring $R$, a finitely generated module $\omega$ is called \emph{canonical} provided that $\omega_\p$ is a canonical module over $R_\p$ for each prime ideal $\p$ of $R$. In this case, if furthermore $\dim(R)<\infty$, then $\id_R(\omega)<\infty$; see, for example, \cite[Theorem 3.1.17]{BH}.

\vskip5pt

Canonical modules are also referred to as dualizing modules in the literature; see, for example, \cite{Auslander-Buchweitz}.

\vskip5pt

As an example, a commutative Artin ring is always a Cohen--Macaulay ring, and the module $J$ as in \Cref{duality} is a canonical module. 

\end{chunk}
\begin{chunk}\label{def-Gor}
 \textbf{(Iwanaga--)Gorenstein rings.} 
A noetherian ring $R$ is called \emph{Iwanaga--Gorenstein} provided that $R$ has finite injective dimension over both sides. By \cite[Lemma A]{Zaks}, if $R$ is Iwanaga--Gorenstein, then $\id_R(R)=\id_{R^{\rm op}}(R)<\infty$.

    \vskip5pt
     
     A commutative noetherian ring $R$ is said to be \emph{Gorenstein} provided that $R_\p$ is Iwanaga--Gorenstein for each prime ideal $\p$ of $R$. For a commutative noetherian ring $R$, if $R$ is Iwanaga--Gorenstein, then $R$ is Gorenstein. The converse holds if, in addition, $R$ is of finite Krull dimension; see \cite[Theorem 3.17]{BH}.
\end{chunk}
\begin{chunk}\label{def-GP}
\textbf{Gorenstein projective (resp. injective) modules.}
An $R$-module $M$ is called \emph{Gorenstein projective} (resp. \emph{Gorenstein injective}) if there exists an acyclic complex of projective (resp. injective) $R$-modules
$$
\mathbf{P}\colon \cdots \longrightarrow P_1 \xrightarrow{d_1} P_0 \xrightarrow{d_0} P_{-1} \longrightarrow \cdots \quad \text{(resp. } \mathbf{I}\colon \cdots \longrightarrow I_{1} \xrightarrow{\partial_{1}} I_0 \xrightarrow{\partial_0} I_{-1} \longrightarrow \cdots\text{)}
$$
such that $\Hom_R(\mathbf{P}, P)$ (resp. $\Hom_R(I, \mathbf{I})$) is acyclic for each $P\in \Proj(R)$ (resp. $I\in \Inj(R)$), and $M$ is isomorphic to the image of $d_0$ (resp.  $\partial_{0}$). 
The complex $\mathbf{P}$ (resp. $\mathbf I$) is called a \emph{totally acyclic complex} of projective (resp. injective) modules. 
We write $\GProj(R)$ (resp. $\GInj(R)$) to be the full subcategory of $\Mod(R)$ consisting of Gorenstein projective (resp. Gorenstein injective) modules.  We refer the reader to \cite[Chapter 10]{Enochs-Jenda} for more details.

\vskip5pt

We write $\Gproj(R)$ to be the full subcategory of $\mo(R)$ consisting of finitely generated Gorenstein projective $R$-modules. 

\end{chunk}

\begin{chunk}\label{def-cotorsion}
    \textbf{Cotorsion pairs.} Let $\A$ be an abelian category. For a subcategory $\mathcal C$ of $\A$, we write
$$
\mathcal C^{\perp_1}\colonequals\{X\in \A\mid \Ext^1_\A(M,X)=0, \forall  M\in \mathcal C\},
$$
$$
\mathcal C^{\perp_\infty}\colonequals\{X\in \A\mid \Ext^i_\A(M,X)=0, \forall i>0, M\in \mathcal C\},
$$
and 
$$
\leftindex^{\perp_1} {\mathcal C}\colonequals \{X\in \A\mid \Ext^1_\A(X,M)=0, \forall M\in \mathcal C\},
$$
$$
\leftindex^{\perp_\infty} {\mathcal C}\colonequals \{X\in \A\mid \Ext^i_\A(X,M)=0, \forall i>0, M\in \mathcal C\}.
$$


It is known that in $\mo(R)$, we have
$$
{\mathcal P}^{<\infty}(R)\subseteq \Gproj(R)^{\perp_\infty} \text{ and } \Gproj(R)\subseteq \leftindex^{\perp_\infty} {\mathcal P}^{<\infty}(R)
$$

  Let $\mathcal X,\mathcal Y$ be full subcategories of $\A$. Following from \cite{Sal}, the pair $(\mathcal X,\mathcal Y)$ is said to be a \emph{cotorsion pair} if $\mathcal X^{\perp_1}=\mathcal Y$ and $\mathcal X=  \leftindex^{\perp_1}{\mathcal Y}$. A cotorsion pair $(\mathcal X, \mathcal Y)$ is said to be \emph{complete} if for any $M\in \mathcal A$, there are two short exact sequences
  \[
    0\rightarrow Y \rightarrow X \rightarrow M \rightarrow 0 \quad\quad  0 \rightarrow M \rightarrow Y' \rightarrow X' \rightarrow 0, 
  \]
  where $X,X'\in \mathcal X$, $Y,Y'\in \mathcal Y$. A cotorsion pair $(\mathcal X, \mathcal Y)$ is said to be \emph{hereditary} if $\mathcal X={}^{\perp_{\infty}}\mathcal Y$, or equivalently, $\mathcal Y=\mathcal X^{\perp_{\infty}}$. We note that if $\A$ contains enough projective objects then a cotorsion pair $(\mathcal X, \mathcal Y)$ is hereditary if and only if $\mathcal X$ is resolving. See more details in \cite[Section 5]{GT}. For example, if $R$ is an Iwanaga--Gorenstein ring, then $(\Gproj(R),{\mathcal P}^{<\infty}(R))$ is a hereditary complete cotorsion pair in $\mo (R)$; see, for example, \cite[Theorem 8.3]{H}. 

  \vskip5pt


\end{chunk}

\begin{chunk}\label{defofvirtuallyG}
\textbf{Virtually Gorenstein rings.}
A noetherian ring $R$ is said to be \emph{virtually Gorenstein} if $ \GProj(R)^{\perp_\infty}= \leftindex^{\perp_\infty} \GInj(R)$. By \cite[Theorem A.1]{Iyengar-Krause2022}, any Gorenstein $A$-algebra in the sense of \cite[Section 4]{IK} is virtually Gorenstein, where $A$ is a commutative noetherian ring. In particular, any commutative Gorenstein ring is virtually Gorenstein. 
\end{chunk}

\begin{chunk}\label{approximation}
    \textbf{Left approximations and covariantly finite subcategories.} Let $\mathcal C$ be a full subcategory of an abelian category $\A$. For each $X\in \mathcal A$ and a morphism $f\colon X\rightarrow C$ with $C\in \mathcal C$, the morphism $f$ is said to be a \emph{left $\mathcal C$-approximation} if any morphism $f^\prime\colon X\rightarrow C^\prime$ with $C^\prime \in \mathcal C$ factors through $f$.  The subcategory $\mathcal C$ is said to be a \emph{covariantly finite} subcategory of $\A$ provided that any object $X\in \A$ has a left $\mathcal C$-approximation. Dually, there are notions of right approximations and contravariantly finite subcategories. 

\vskip5pt

    Let $(\mathcal X,\mathcal Y)$ be a complete cotorsion pair of $\mathcal A$. Then, $\mathcal X$ is contravariantly finite and $\mathcal Y$ is covariantly finite. 

\vskip5pt

    Let $M$ be an $R$-module and $\add M$ be the full subcategory of $\mo(R)$ consisting of modules that are direct summands of finite direct sums of $M$. When $R$ is a noetherian algebra, it is known that $\add M$  is both covariantly finite and contravariantly finite in $\mo(R)$; see, for example, \cite[Proposition 4.2]{AS}.
\end{chunk}

\begin{chunk}\textbf{A criterion for Gorenstein projective modules. }\label{complete resolution}
    Let $R$ be a noetherian ring and $M$ be a finitely generated $R$-module. Assume that $\Ext^i_R(M,R)=0$ for each $i>0$ and there exists an exact sequence of $R$-modules
    \[
0 \rightarrow M\xrightarrow{d_0} P_{-1} \xrightarrow{d_{-1}} P_{-2}\xrightarrow{d_{-2}}\cdots\rightarrow P_{-n}\rightarrow \cdots,
\]
where $P_{-i}\in \proj(R)$ and the natural embedding $\Ima d_{-(i-1)}\rightarrow P_{-i}$ is a left $\proj(R)$-approximation for each $i\geq 1$. Then $M\in \Gproj(R)$.

\vskip5pt
Indeed, splicing a projective resolution of $M$ with the projective coresolution of $M$ constructed above produces a totally acyclic complex of finitely generated projective $R$-modules. Consequently, $M$ is Gorenstein projective.
\end{chunk}

\section{Characterizations of Gorenstein projective modules}
The main result of this section is \Cref{T2} from the introduction; see \Cref{cotorsion pair} and \Cref{takanswer1}. They together with \Cref{characterization} and \Cref{char-GP-Artin} provide a description of the finitely generated Gorenstein projective modules for two certain classes of rings. 


\vskip5pt

The following lemma is well-known. 

\begin{lemma}\label{thick subcategory}
    Let $R$ be a noetherian ring. In $\mo(R)$, there are 
    \begin{enumerate}
        \item $\Gproj(R)^{\perp_1}=\Gproj(R)^{\perp_\infty}$.
        \item $\Gproj(R)^{\perp_1}$ is a thick subcategory.
    \end{enumerate}
\end{lemma}
\begin{proof}
    $(1)$ 
   This follows from the fact that $\Gproj(R)$ is a resolving subcategory of $\mo(R)$. $(2)$ Combining with (1), the statement follows from \cite[Proposition 4.4]{Takahashi}.
\end{proof}

\begin{corollary}\label{fid contains}
   For a noetherian ring $R$, there is an inclusion in $\mo(R)$
   $$
   \Thick(\proj(R)\cup {\mathcal I}^{<\infty}(R))\subseteq \Gproj(R)^{\perp_1}.
   $$ 
\end{corollary}
\begin{proof}
    Note that in $\Mod(R)$, for $G\in \GProj(R)$ and $X\in \proj(R)\cup {\mathcal I}^{<\infty}(R)$, there is $\Ext_R^1(G, X)=0$. It follows that $\proj(R)\cup {\mathcal I}^{<\infty}(R)\subseteq\Gproj(R)^{\perp_1}$ in $\mo(R)$. The desired result now follows from \Cref{thick subcategory} (2). 
\end{proof}

\begin{lemma}\label{global version}
    Let $S$ be a Cohen--Macaulay ring that admits a canonical module $\omega$. Then
\begin{enumerate}
    \item $\MCM(S)=\leftindex^{\perp_\infty}\omega$ in $\mo(S)$. 
    \item For $M\in \MCM(S)$, the canonical map $M\rightarrow \Hom_S(\Hom_S(M,\omega),\omega);~x\mapsto (f\mapsto f(x))$ is an isomorphism. 
\end{enumerate}
\end{lemma}
\begin{proof}
    $(1)$ First, we show that $\MCM(S)\subseteq\leftindex^{\perp_\infty}\omega$. It suffices to show that for each $M\in \MCM(S)$, prime ideal $\p$ and $i\geq 1$, there is $\Ext_{S_\p}^i(M_\p, \omega_\p)=0$. Since $S_\p$ is a Cohen--Macaulay local ring that admits a canonical module $\omega_\p$ and $M_\p\in \MCM(S_\p)$, one has $\Ext_{S_\p}^i(M_\p, \omega_\p)=0$ by \cite[Theorem 3.3.10 (d) (ii)]{BH}. 

\vskip5pt

    Conversely, suppose that $M\in \leftindex^{\perp_\infty}\omega$. Since $S$ is a Cohen--Macaulay ring, it suffices to show that for each prime ideal $\p$ such that $M_\p\ne 0$, there is $\depth(M_\p)=\depth(S_\p)$. Since $\Ext_{S_p}^i(M_\p, \omega_\p)=0$ and $\id_{S_\p} (\omega_\p)<\infty$ (see \Cref{def of canonical}), $\depth(M_\p)=\depth(S_\p)$ by \cite[3.1.24]{BH}. 

\vskip5pt

    $(2)$ By \cite[Theorem 3.3.10 (d) (iii)]{BH}, for $M\in \MCM(S)$ and any prime ideal $\p$, the canonical map $M_\p\rightarrow \Hom_{S_\p}(\Hom_{S_\p}(M_{\p},\omega_{\p}),\omega_{\p})$ is an isomorphism. Thus, the canonical map $M\rightarrow \Hom_S(\Hom_S(M,\omega),\omega)$ is an isomorphism. 
\end{proof}

\begin{lemma}\label{monomorphism}
    Let $S$ be a Cohen--Macaulay ring that admits a canonical module $\omega$, and $R$ a noetherian $S$-algebra. For each $M\in\mo(R)$, if $M$ is maximal Cohen--Macaulay as an $S$-module, then there is an exact sequence of $R$-modules
    \[
     0 \rightarrow M \rightarrow (R^\dagger)^{m_0} \rightarrow (R^\dagger)^{m_1} \rightarrow \cdots \rightarrow (R^\dagger)^{m_n} \rightarrow \cdots, 
    \]
    where each $m_i\geq 0$ and $(-)^\dagger\colonequals \Hom_S(-,\omega)$. In particular, there is a monomorphism $M\rightarrow ({R}^\dagger)^{m_0}$ in $\mo(R)$. 
\end{lemma}
\begin{proof}
Choose a free resolution of $M^\dagger$ in $\mo(R^{\rm op})$
\[
\cdots \rightarrow R^{m_n}\rightarrow \cdots \rightarrow R^{m_1}\rightarrow R^{m_0}\rightarrow M^\dagger\rightarrow 0,
\]
Since $M$ is maximal Cohen--Macaulay as an $S$-module, $M^\dagger$ is also maximal Cohen--Macaulay as an $S$-module by \cite[Theorem 3.3.10 (d) (i)]{BH}. Also, $\Ext_S^i(M^\dagger, \omega)=0$ for $i>0$; see \Cref{global version}. Thus, applying $\Hom_S(-,\omega)$, one gets an exact sequence
 \[
     0 \rightarrow M^{\dagger\dagger} \rightarrow (R^\dagger)^{m_0} \rightarrow (R^\dagger)^{m_1} \rightarrow \cdots \rightarrow (R^\dagger)^{m_n} \rightarrow \cdots, 
 \]
  By \Cref{global version}, $M\cong M^{\dagger\dagger}$ as $R$-modules. This completes the proof.
\end{proof}

\begin{lemma}\label{change of ring-iso}
     Let $S$ be a Cohen--Macaulay ring that admits a canonical module $\omega$. Assume that $R$ is a noetherian $S$-algebra and that $R$, viewed as an $S$-module, is maximal Cohen--Macaulay. Then, for each $M\in \mo(R)$ and $i\geq 0$, there is an isomorphism
     $$ 
     \Ext^i_S(M,\omega)\cong \Ext^i_R(M,{R}^\dagger).
     $$
\end{lemma}

\begin{proof}
Since $R\in \MCM(S)$, the complex $\RHom_S(R, \omega)$ is quasi-isomorphic to $ R^\dagger$, and hence the isomorphism follows from the derived tensor–hom adjunction. Alternatively, it can also be seen directly from the proof of \cite[Theorem 1.2 (b)]{M}.
\end{proof}
 
\begin{theorem}\label{characterization}
 Let $S$ be a Cohen--Macaulay ring that admits a canonical module $\omega$. Assume that $R$ is a noetherian $S$-algebra and that $R$, viewed as an $S$-module, is maximal Cohen--Macaulay. Then, in $\mo(R)$, 
 $$
 \Gproj(R)= \leftindex^{\perp_1}\Thick(\proj(R)\cup \{R^\dagger\}). 
 $$
\end{theorem}

\begin{proof}
Let $G\in \Gproj(R)$, and 
\[
\cdots \longrightarrow P_1 \xrightarrow{d_1} P_0 \xrightarrow{d_0} P_{-1} \longrightarrow \cdots
\]
the associating totally acyclic complex of finitely generated projective $R$-modules, where $G=\Ima d_0$. First, we show that $G$ is maximal Cohen--Macaulay as an $S$-module. By definition, it suffices to show that, for each prime ideal $\p$ of $S$ such that $G_\p\ne 0$, there is $\depth (G_\p)=\dim(S_\p)$. Since the localization functor is exact, there is a long exact sequence
\[
\cdots \longrightarrow {(P_1)}_\p \xrightarrow{\overline{d_1}} {(P_0)}_\p \xrightarrow{\overline{d_0}} {(P_{-1})}_\p \longrightarrow \cdots, 
\]
where $\Ima \overline{d_0}=G_\p$. By assumption, $R\in \MCM(S)$. It follows that all $P_i\in \MCM(S)$ and then all $(P_i)_\p\in \MCM(S_\p)$, $i\in \Z$. That is, $\depth ((P_i)_\p)=\dim (S_\p)$, $i\in \Z$. By \cite[3.1.24]{BH} and the fact that $S_\p$ admits a finite module $\omega_\p$ of finite injective dimension, for any $M\in \mo(S_\p)$ there is $\depth (M)\leq \depth (S_\p)$. And since $S_\p$ is a Cohen--Macaulay ring, $\depth (S_\p)=\dim(S_\p)$. Now consider $t\colonequals\min \{\depth (\Ima\overline{d_i}) \mid i\in \Z\}$ and suppose that $\depth (\Ima\overline{d_k})=t$. By \cite[Proposition 1.2.9]{BH}, 
\[
\begin{aligned}
    \depth (\Ima \overline{d_{k}})&\geq \min\{\depth ((P_{k-1})_\p), \depth(\Ima \overline{d_{k-1}})+1\}\\
    &=\min\{\dim (S_\p), \depth (\Ima \overline{d_{k-1}})+1\}.
\end{aligned}
\]
But by the choice of $k$, 
\[
\depth (\Ima \overline{d_{k}})< \depth (\Ima \overline{d_{k-1}})+1.
\]
Therefore, $\dim (S_\p)\leq \depth (\Ima \overline{d_{k}})\leq \depth (\Ima \overline{d_{0}})=\depth (G_\p)\leq \dim (S_\p)$, implying that $\depth (G_\p)= \dim (S_\p)$. 

\vskip5pt

By the above, $G$ is maximal Cohen--Macaulay as an $S$-module. It follows that $\Ext_S^i(G, \omega)=0$; see \Cref{global version} (1). Then, by \Cref{change of ring-iso}, 
$$\Ext_R^1(G, R^\dagger)=\Ext_S^1(G, \omega)=0.$$ 
It follows that $R^\dagger\in \Gproj(R)^{\perp_1}$. It is also known that $\proj(R)\in \Gproj(R)^{\perp_1}$. Thus, by \Cref{thick subcategory} (2), $\Thick(\proj(R)\cup\{ R^\dagger\})\subseteq \Gproj(R)^{\perp_1}$, which is equivalent to say that $\Gproj(R)\subseteq \leftindex^{\perp_1}{\Thick(\proj(R)\cup \{R^\dagger\})}$. 


\vskip5pt

Conversely, let $X\in \leftindex^{\perp_1}\Thick(\proj(R)\cup \{R^\dagger\})$. Suppose that $f_0\colon X\rightarrow Q_{-1}$ is a left $\proj(R)$-approximation, where $Q_{-1}\in\proj(R)$; see \Cref{approximation}. We claim that the morphism $f_0$ is a left $\Thick(\proj(R)\cup \{R^\dagger\})$-approximation. For any $g\colon X\rightarrow Y$ with $Y\in \Thick(\proj(R)\cup \{R^\dagger\})$, take a short exact sequence 
\[
\xymatrix{
0\ar[r] & Y'\ar[r] & Q\ar[r]^-{\pi} & Y\ar[r] & 0
}
\]
where $Q\in \proj(R)$. Note that $g$ factors through $\pi$ as $X\in \leftindex^{\perp_1}\Thick(\proj(R)\cup \{R^\dagger\})$ and $Y'\in \Thick(\proj(R)\cup \{R^\dagger\})$, say $g=\pi h$ for some morphism $h\colon X\rightarrow Q$. Since $f_0$ is a left $\proj(R)$-approximation, $h$ factors through $f_0$, and hence $g$ factors through $f_0$. Thus, $f_0$ is a left $\Thick(\proj(R)\cup \{R^\dagger\})$-approximation. 

\vskip5pt

Next, we show that $f_0$ is a monomorphism. By \cite[Theorem 1.2 (c)]{M}, there is a short exact sequence of $R$-modules
\[
\xymatrix{
0\ar[r] & K\ar[r] & M\ar[r]^-{\pi} & X\ar[r] & 0, 
}
\]
where $K\in \Thick(\proj(R)\cup R^\dagger)$, $M$ is maximal Cohen--Macaulay as an $S$-module. Since $X\in \leftindex^{\perp_1}\Thick(\proj(R)\cup \{R^\dagger\})$, this exact sequence splits and $X$ is a direct summand of $M$. In particular, $X$ is maximal Cohen--Macaulay as an $S$-module. By \Cref{monomorphism}, there is a monomorphism $t\colon X\rightarrow (R^\dagger)^n$ for some positive integer $n$. Then, $t$ factors through $f_0$ since $f_0$ is a left $\Thick(\proj(R)\cup \{R^\dagger\})$-approximation. It follows that $f_0$ is a monomorphism. 


\vskip5pt

By the above, $f_0$ is a left $\Thick(\proj(R)\cup \{R^\dagger\})$-approximation and is a monomorphism. Combining this with that $X\in \leftindex^{\perp_1} \Thick(\proj(R)\cup\{R^\dagger\})$, we conclude that $\Coker f_0\in \leftindex^{\perp_1}\Thick(\proj(R)\cup \{R^\dagger\})$. Then, by induction, there is an exact sequence 
\[
0 \rightarrow X\xrightarrow {f_0} Q_{-1} \xrightarrow {f_{-1}} Q_{-2}\xrightarrow{f_{-2}}\cdots\rightarrow Q_{-n}\rightarrow \cdots
\]
where all $Q_{-i}\in \proj(R)$ and the natural embedding $\Ima f_{-(i-1)}\rightarrow Q_{-i}$ is a left $\Thick(\proj(R)\cup \{R^\dagger\})$-approximation for each $i\geq 1$. In particular, the natural embedding $\Ima f_{-(i-1)}\rightarrow Q_{-i}$ is a left $\proj(R)$-approximation for each $i\geq 1$

\vskip5pt

In order to show that $X\in\Gproj(R)$, by \Cref{complete resolution}, it remains to show that $\Ext_R^i(X, R)=0$ for $i>0$. By \Cref{monomorphism}, there is an exact sequence of $R$-modules
\[
     0 \rightarrow R \rightarrow (R^\dagger)^{m_0} \xrightarrow{\phi_0} (R^\dagger)^{m_1} \xrightarrow{\phi_1} \cdots \xrightarrow{\phi_{n-1}} (R^\dagger)^{m_n} \xrightarrow{\phi_n} \cdots, 
\]
where  $m_i\geq 0$ and $\Ima \phi_i\in \Thick(\proj(R) \cup \{R^\dagger\})$ for each $i\geq 0$. Since $X\in \leftindex^{\perp_1}\Thick(\proj(R) \cup \{R^\dagger\})$, $\Ext_R^1(X, R)=0$ . As shown above, $X$ is maximal Cohen--Macaulay as an $S$-module. It follows that $\Ext_R^i(X, R^\dagger)=\Ext_S^i(X, \omega)=0$ for $i>0$ by \Cref{change of ring-iso} and \Cref{global version} (1). Thus, 
\[
\Ext_R^i(X, R)=\Ext_R^{i-1}(X, \Ima \phi_0)=\cdots=\Ext_R^1(X, \Ima\phi_{i-2})=0
\]
for $i\geq 2$. This finishes the proof.
\end{proof}

\begin{remark}\label{infty setting}
 Under the assumptions of \Cref{characterization}, we have
      $$
       \Gproj(R)=\leftindex^{\perp_\infty}\Thick(\proj(R)\cup \{R^\dagger\})= \leftindex^{\perp_1}\Thick(\proj(R)\cup \{R^\dagger\}). 
      $$
      Indeed, by \Cref{characterization}, we have the inclusion $$\Thick(\proj(R)\cup \{R^\dagger\})\subseteq \Gproj(R)^{\perp_1}=\Gproj(R)^{\perp_\infty},$$
      where the equality is by \Cref{thick subcategory}. This yields the first inclusion below
      $$
      \Gproj(R)\subseteq \leftindex^{\perp_\infty}\Thick(\proj(R)\cup \{R^\dagger\})\subseteq \leftindex^{\perp_1}\Thick(\proj(R)\cup \{R^\dagger\}).
      $$
      The desired statement now follows from \Cref{characterization}.
\end{remark}

Here are some natural examples of noetherian $S$-algebras $R$ that satisfy the assumptions in \Cref{characterization}. 

\begin{example}\label{Examples}
(1) Let $S$ be a Cohen--Macaulay ring that admits a canonical module. Then $S$ itself satisfies the assumptions in \Cref{characterization}; see \Cref{mcm}. In particular, a complete Cohen--Macaulay local ring satisfies the assumptions; see \cite[Corollary 3.3.8]{BH}. 

\vskip5pt

(2) Let $S$ be a commutative Artin ring, and $R$ a module finite $S$-algebra. Then $R$ satisfies the assumptions in \Cref{characterization}; see \Cref{duality}. 

\vskip5pt

(3) Let $S$ be a Cohen--Macaulay ring that admits a canonical module. Assume that $S \to R$ is a Frobenius extension in the sense of \cite[Theorem~1.2]{Kadison}. For instance, one may take $R=S\llbracket x\rrbracket/(x^2)$. Then $R$ satisfies the assumptions in \Cref{characterization}.

\vskip5pt

(4) Let $S$ be a Cohen--Macaulay ring that admits a canonical module and $M\in \MCM(S)$. Assume $R=S\ast M$ is the trivial extension of $S$ by $M$; see details in \cite[Theorem 3.3.6]{BH}. Then $R$ satisfies the assumptions in \Cref{characterization}.

\vskip5pt

(5) If $S$ is a commutative local Cohen--Macaulay ring of dimension $\le 2$ and $M\in \MCM(S)$, then $R:=\End_S(M)$ satisfies the assumptions in \Cref{characterization}; see \cite[Theorem 3.3.10]{BH}. 
\end{example}

\begin{corollary}\label{ortho of proj-fid}
     Keep the assumptions as in \Cref{characterization}. If, in addition, $\dim(S)<\infty$, then, in $\mo (R)$, we have 
    $$
 \Gproj(R)=\leftindex^{\perp_1}\Thick(\proj(R)\cup {\mathcal I}^{<\infty}(R)).
    $$
\end{corollary}
\begin{proof}
    Assume $\dim(S)<\infty$. By \Cref{def of canonical}, $\id_S(\omega)<\infty$. It follows from the argument of \Cref{change of ring-iso} that $\id_R(R^\dagger)<\infty$. This yields the following inclusion
    $$
    \leftindex^{\perp_1}\Thick(\proj(R)\cup {\mathcal I}^{<\infty}(R))\subseteq \leftindex^{\perp_1}\Thick(\proj(R)\cup \{R^\dagger\})=\Gproj(R),
    $$
    where the equality is by \Cref{characterization}. The desired result now follows by combining with \Cref{fid contains}.
\end{proof}

\begin{corollary}\label{cotorsion pair}
    Under the assumptions of \Cref{characterization},  there is a hereditary cotorsion pair $(\Gproj(R),\Gproj(R)^{{\perp}_1})$  in $\mo(R)$.
\end{corollary}
\begin{proof}
    By \Cref{characterization}, $\Gproj(R)=\leftindex^{\perp_1}\Thick(\proj(R)\cup \{R^\dagger\})$. Thus, 
    \[
    \begin{aligned}
    \leftindex^{\perp_1}{(\Gproj(R)^{\perp_1})}&= \leftindex^{\perp_1}{((\leftindex^{\perp_1} {\Thick(\proj(R)\cup \{R^\dagger\})})^{\perp_1})}\\
    &= \leftindex^{\perp_1}\Thick(\proj(R)\cup \{R^\dagger\}) \\
    &=  \Gproj(R).
     \end{aligned}
    \] 
    It follows that $(\Gproj(R),\Gproj(R)^{{\perp}_1})$ is a cotorsion pair in $\mo(R)$. 

\vskip5pt

    For the heredity, it is known that $\Gproj(R)$ is closed under the kernel of epimorphism, and $\Gproj(R)^{{\perp}_1}$ is closed under the cokernel of monomorphism by \Cref{thick subcategory} (2). 
\end{proof}

\begin{remark}\label{big version-cotorsion}
  (1)  The cotorsion pair $(\Gproj(R),\Gproj(R)^{{\perp}_1})$ may fail to be complete. Indeed, $\Gproj(R)$ may fail to be contravariantly finite in $\mo(R)$. Y. Yoshino (\cite[Theorem 6.1]{Yoshino2003}) showed that there is a class of Artin algebras $\Lambda$ such that $\Gproj(\Lambda)$ is not contravariantly finite in $\mo(\Lambda)$; and a precise example was given by A. Beligiannis and H. Krause in \cite[Proposition 4.3]{BK}. 

\vskip5pt
It is worth recalling that for a commutative noetherian henselian local ring, L.W. Christensen,
G. Piepmeyer, J. Striuli, and R. Takahashi \cite[(2.2) and Theorem~C]{CPST} proved the following remarkable
result: if $\Gproj(R)$ is contravariantly finite in
$\mo(R)$, then $R$ is either Gorenstein or $\Gproj(R)=\proj(R)$.



\vskip5pt
(2) Let $R$ be an arbitrary unital ring. M. Cortés-Izurdiaga and J. Šaroch \cite[Corollary 3.4 (1)]{CIS} observed that the pair \((\GProj(R), \GProj(R)^{\perp_1})\) forms a hereditary cotorsion pair in \(\Mod(R)\). The corresponding statement for Artin algebras was established earlier by A. Beligiannis and I. Reiten \cite[Chapter X, Theorem 2.4]{Beligiannis-Reiten}. 
\Cref{cotorsion pair} can be viewed as an analogue of this result in the category of finitely generated modules.
\vskip5pt
\end{remark}

The following result is analogous to \Cref{characterization}. It strengthens the result mentioned in \Cref{big version-cotorsion} (2).
\begin{proposition}
    Let $R$ be an arbitrary unital ring and $\Thick ^\Pi(\Proj(R) \cup \Inj(R))$ be the smallest thick subcategory of $\Mod(R)$ containing all the direct products of projective $R$-modules and injective $R$-modules.  Then $$\GProj(R)= \leftindex^{\perp_1}\Thick ^\Pi(\Proj(R) \cup \Inj(R)).$$ 
\end{proposition}
\begin{proof}
    Since $\Thick ^\Pi(\Proj(R) \cup \Inj(R))\subseteq \GProj(R)^{\perp_1}$, we have $$\GProj(R)=\leftindex^{\perp_1}(\GProj(R)^{\perp_1})\subseteq \leftindex^{\perp_1}\Thick ^\Pi(\Proj(R) \cup \Inj(R)).$$  Conversely, let $X\in \leftindex^{\perp_1}\Thick ^\Pi(\Proj(R) \cup \Inj(R))$. By \cite[Lemma 3.1]{CIS}, there is a short exact sequence
       $$  
        \xymatrix{
        0\ar[r] & X\ar[r]^-{f} & P_1\ar[r] & X_1\ar[r] & 0,
        }
       $$
        where $P_1\in \Proj(R)$, $X_1\in \leftindex ^{\perp_1}\Thick ^\Pi(\Proj(R) \cup \Inj(R))$, and the map $f$ is a left $\Thick ^\Pi(\Proj(R) \cup \Inj(R))$-approximation. Then a similar argument as in the proof of \Cref{characterization} will yield that $X$ is Gorenstein projective. 
\end{proof}

Next, we focus on proving that for maximal Cohen--Macaulay algebras $R$ over  Cohen--Macaulay local rings, $(\Gproj(R), \Gproj(R)^{\perp_1})$ is a hereditary cotorsion pair in $\mo (R)$. 

\begin{lemma}\label{extorthmcm} 
Let $(S,\m, k)$ be a Cohen--Macaulay local ring. Assume that $R$ is a noetherian $S$-algebra and that $R$, viewed as an $S$-module, is maximal Cohen--Macaulay.
For a module $M\in \mo(R)$, if $\Ext_R^1(M,N)=0$ for every $N\in\mo(R)$ with finite injective dimension, then $M\in\MCM(S)$. 
\end{lemma}

\begin{proof} 
Assume $\dim(S)=d$. Choose a maximal $S$-regular sequence $x_1,...,x_d\in \m$. Let $E_S(S/\m)$ denote the injective envelope of the module $S/\m$ over $S$. Set $N:=\Hom_S(R_R/(x_1,...,x_d)R_R, E_S(S/\m))$. Note that $N$ has finite length over $S$, and hence $N\in\mo(R)$.  Moreover, $N$ has finite injective dimension over $R$. Indeed, let $K$ denote the Koszul complex on the sequence $x_1,\ldots,x_d$ over $S$. Note that $x_1,\ldots,x_d$ is also $R$-regular as $R\in\MCM(S)$. Hence,
there is a quasi-isomorphism
$$K\otimes_S R\xrightarrow \simeq R_R/(x_1,...,x_d)R_R;$$
see \cite[Corollary 1.6.14]{BH}. Applying the exact functor $\Hom_S(-,E_S(S/\m))$ on this quasi-isomorphism, we could get a bounded injective resolution of $N$ over $R$. 

\vskip5pt

By the above and the hypothesis, $\Ext^1_R(M, N)=0$. Combining this with the isomorphism $\Hom_S(N, E_S(S/\m))\cong R_R/(x_1,...,x_d)R_R$ (see \cite[Proposition 3.2.12]{BH}), we have $\Tor^R_1(R_R/(x_1,...,x_d)R_R, M)=0 $; see, for instance, \cite[Remark 4.7]{cff}. Note that
\begin{align*}
    \Tor^R_1(R_R/(x_1,...,x_d)R_R, M)& \cong \h_1((K\otimes_S R)\otimes_R M)\\
    & \cong \h_1(K\otimes_S M).
\end{align*}
Thus, $\h_1(K\otimes_S M)=0$.
By \cite[Corollary 1.6.19]{BH}, $x_1,...,x_d$ is $M$-regular.   Hence, $M\in\MCM(S)$. 
\end{proof}
The following lemmas are used in the proof of \Cref{takanswer1} or \Cref{Tor-pair}. All of them follow from the same arguments as in the commutative case; see \cite[Lemma~1.4.4]{Christensen}, \cite[Lemma 2 (ii) in Section 18]{matsumura} and  \cite[Lemma 2 (iii) in Section 18]{matsumura}. 
\begin{lemma}\label{reduction}\label{reduction2}
    Let $(S,\m, k)$ be a commutative local ring and $R$ be a noetherian $S$-algebra. Assume that $M\in\mo(R)$ and $x\in \m$ is $R$-regular. 
    
    \vskip5pt
    
    (1) If $x$ is also $M$-regular, then $M\in \Gproj(R)$ if and only if $M/xM\in \Gproj(R/xR)$. 
    
    \vskip5pt
    
    (2) If $x$ is also $M$-regular, then $\Ext_R^1(M, N)\cong\Ext_{R/xR}^1(M/xM, N)$ for each $N\in \mo(R)$ with $xN=0$. 

    \vskip5pt

    (3) If $xM=0$, then $\Tor_1^R(N, M)\cong\Tor_1^{R/xR}(N/xN, M)$ for each $N\in \mo(R^{\rm op})$ such that $x$ is $N$-regular. 
\end{lemma}

\begin{lemma}\label{regperp1}
Let $(S,\m, k)$ be a Cohen--Macaulay local ring and $x\in \m$ be an $S$-regular element. Assume that $R$ is a noetherian $S$-algebra and that $R$, viewed as an $S$-module, is maximal Cohen--Macaulay.
If $M\in \leftindex^{\perp_1}(\Gproj(R)^{\perp_1})$, then $M\in\MCM(S)$ and $M/xM\in \leftindex^{\perp_1}(\Gproj(R/xR)^{\perp_1})$. 
\end{lemma}

\begin{proof}
For every finitely generated $R$-module $N$ of finite injective dimension, we have $N\in \Gproj(R)^{\perp_1}$, and hence $\Ext_R^{1}(M,N)=0$ by hypothesis. Combining this with \Cref{extorthmcm}, we get that $M\in\MCM(S)$. Hence, $x$ is also $M$-regular. 

  \vskip5pt
Let $L\in \Gproj(R/xR)^{\perp_1}\subseteq \mo(R/xR)$. To show that $\Ext_{R/xR}^{1}(M/xM,L)=0$, we first show that $L\in \Gproj(R)^{\perp_1}$ as an $R$-module. Pick arbitrary $G\in \Gproj(R)$. Combining this with $R\in\MCM(S)$, we have $G\in\MCM(S)$ as the same statement in the proof of \Cref{characterization}. Combining with $x$ is $S$-regular, we get that $x$ is both $G$-regular and $R$-regular. By \Cref{reduction} (1), $G/xG\in \Gproj(R/xR)$, 
 and hence $\Ext^{1}_{R/xR}(G/xG, L)=0$. Since $L\in \mo(R/xR)$, one has $\Ext_R^{1}(G, L)=0$ by \Cref{reduction2} (2). Namely, $L\in \Gproj(R)^{\perp_1}$. Hence, the assumption on $M$ yields that $\Ext_R^{1}(M, L)=0$. Again by \Cref{reduction2} (2), there is $\Ext_{R/xR}^{1}(M/xM, L)=0$. 
\end{proof}

\begin{theorem}\label{takanswer1} Let $(S,\m, k)$ be a Cohen--Macaulay local ring. Assume that $R$ is a noetherian $S$-algebra and that $R$, viewed as an $S$-module, is maximal Cohen--Macaulay. Then, in $\mo(R)$, $$\Gproj(R)=\leftindex^{\perp_\infty}(\Gproj(R)^{\perp_\infty})=\leftindex^{\perp_1}(\Gproj(R)^{\perp_1})$$ In particular, $(\Gproj(R), \Gproj(R)^{\perp_1})$ is a hereditary cotorsion pair in $\mo(R)$. 
\end{theorem}

\begin{proof}

We know by \Cref{thick subcategory} (1) that $\Gproj(R)^{\perp_1}=\Gproj(R)^{\perp_\infty}$, and hence, $$\Gproj(R) \subseteq \leftindex^{\perp_\infty}(\Gproj(R)^{\perp_\infty})\subseteq \leftindex^{\perp_1}(\Gproj(R)^{\perp_1})$$ It remains to prove $\leftindex^{\perp_1}(\Gproj(R)^{\perp_1})\subseteq \Gproj(R)$. For each $M$ in $\leftindex^{\perp_1}(\Gproj(R)^{\perp_1})$, next, we show that $M\in\Gproj(R)$.

      \vskip5pt
      
    Choose a maximal $S$-regular sequence $x_1,\ldots,x_d$ in $\m$, where $d=\dim (S)$. By \Cref{regperp1}, $M\in\MCM(S)$ and 
    $$M/(x_1,\ldots,x_d)M\in \leftindex^{\perp_1}(\Gproj(R/(x_1,\ldots,x_d)R)^{\perp_1}).$$
    Note that $R/(x_1,\ldots,x_d)R$ is an Artin algebra. By \Cref{cotorsion pair}, $$\Gproj(R/(x_1,...,x_d)R)=\leftindex^{\perp_1}(\Gproj(R/(x_1,\ldots,x_d)R)^{\perp_1}).$$ It follows that $M/(x_1,\ldots,x_d)M\in \Gproj(R/(x_1,\ldots,x_d)R)$, and hence $M\in \Gproj(R)$ by \Cref{reduction} (1). 
\end{proof}

\begin{remark}\label{Support}
 Let $R$ be a commutative Noetherian ring. In \cite[Question 4.7]{Takahashi}, R. Takahashi asked whether
\[
\Gproj(R)=\leftindex^{\perp_\infty}(\Gproj(R)^{\perp_\infty})
\]
always holds.

  \vskip5pt
  
Thus, \Cref{cotorsion pair} and \Cref{takanswer1} show that the question has an affirmative answer for certain two classes of algebras $R$; in particular, it holds for a Cohen--Macaulay local ring. This significantly improves \cite[Corollary 4.6]{Takahashi}, which states that the question holds when $R$ is a generically Gorenstein Cohen--Macaulay local ring admitting a canonical module.
\end{remark}

\begin{chunk}
   Let $\mathcal C$ be a subcategory of $\mo(R)$ (resp. $\mo(R^{\rm op})$). We write $\leftindex^{\top_1}{\mathcal C}$ (resp. $\mathcal C^{\top_1}$) to be the full subcategory of $\mo(R^{\rm op})$ (resp. $\mo(R)$) consisting of all modules $X$ satisfying $\Tor_1^R(X,C)=0$ (resp. $\Tor_1^R(C,X)=0$) for all $C\in\mathcal C$. Similarly, we write $\leftindex^{\top_\infty}{\mathcal C}$ (resp. $\mathcal C^{\top_\infty}$) to be the full subcategory of $\mo(R^{\rm op})$ (resp. $\mo(R)$) consisting of all modules $X$ satisfying $\Tor_i^R(X,C)=0$ (resp. $\Tor_i^R(C,X)=0$) for all $C\in\mathcal C$ and all $i>0$. 
\end{chunk}

Similar to \Cref{def-cotorsion}, one can define Tor pairs. See details of Tor pairs in \cite[Section 5]{GT}. In \Cref{Tor-pair}, we prove a version of \Cref{takanswer1} for Tor pairs. Before this, we begin with some preliminary results. 

\begin{corollary}\label{torpairartin}
   Let $\Lambda$ be an Artin algebra with the duality functor $D$. Then $\Gproj(\Lambda^{\rm op})$ is precisely the $\Tor$-orthogonal class of $\Thick_{\mo(\Lambda)}(\{\Lambda, D(\Lambda_\Lambda)\})$. That is, 
    \[
    \Gproj(\Lambda^{\rm op})=^{\top_{1}}\Thick_{\mo(\Lambda)}(\{\Lambda, D(\Lambda_\Lambda)\})=^{\top_{\infty}}\Thick_{\mo(\Lambda)}(\{\Lambda, D(\Lambda_\Lambda)\}). 
    \]
\end{corollary}
\begin{proof}
    Note that $D(\Thick_{\mo(\Lambda^{\rm op})}(\{\Lambda_\Lambda, D(\Lambda)\}))=\Thick_{\mo(\Lambda)}(\{\Lambda, D(\Lambda_\Lambda)\})$. Thus, the assertion follows by \Cref{characterization} and the natural isomorphism 
    \[
    \Ext_{\Lambda^{\rm op}}^i(-, D(M))\cong D(\Tor_i^{\Lambda}(-, M)) 
    \]
    for each $i>0$ and $M\in\mo(\Lambda)$. 
\end{proof}

\begin{lemma}\label{tormcm} Let $(S,\m, k)$ be a Cohen--Macaulay local ring and $x\in \m$ be an $S$-regular element. Assume that $R$ is a noetherian $S$-algebra and that $R$, viewed as an $S$-module, is maximal Cohen--Macaulay. 
If $M\in \leftindex^{\top_1}(\Gproj(R^{\rm op})^{\top_1})$, then $M\in\MCM(S)$ and $M/xM\in \leftindex^{\top_1}(\Gproj(R^{\rm op}/xR^{\rm op})^{\top_1})$. 
\end{lemma}

\begin{proof}
Assume $\dim(S)=d$. Choose a maximal $S$-regular sequence $x_1,...,x_d\in \m$. Note that $x_1,\ldots,x_d$ is also $R$-regular as $R\in\MCM(S)$. For every finitely generated $R$-module $N$ of finite projective dimension, we have $N\in \Gproj(R^{\rm op})^{\top_1}$. In particular, $R/(x_1,...,x_d)R\in \Gproj(R^{\rm op})^{\top_1}$ and $\Tor^R_1(M, R/(x_1,...,x_d)R)=0$. Thus, $M\in\MCM(S)$ by the same argument as in the last part of the proof of \Cref{extorthmcm}. For the second statement, the proof is similar to \Cref{regperp1}. 
\end{proof}

\begin{proposition}\label{Tor-pair}
    Let $(S,\m, k)$ be a Cohen--Macaulay local ring. Assume that $R$ is a noetherian $S$-algebra and that $R$, viewed as an $S$-module, is maximal Cohen--Macaulay. Then 
    $$\Gproj(R^{\rm op})={}^{\top_{\infty}}(\Gproj(R^{\rm op})^{\top_{\infty}})={}^{\top_{1}}(\Gproj(R^{\rm op})^{\top_{1}}).$$ 
\end{proposition}

\begin{proof}
    Since $\Gproj(R^{\rm op})$ is a resolving subcategory of $\mo (R^{\rm op})$, it follows that $$\Gproj(R^{\rm op})^{\top_{\infty}}=\Gproj(R^{\rm op})^{\top_{1}}.$$
    The remainder of the proof is similar to that of \Cref{takanswer1} by using \Cref{tormcm} in place of \Cref{regperp1} and using \Cref{torpairartin} in place of \Cref{cotorsion pair}. 
\end{proof}

Given a  resolving subcategory $\mathcal X$ of $\mo(R)$, the following examples show that $(\mathcal X, \mathcal X^{\perp_1})$ does not form a cotorsion pair nor $(\X,\X^{\top_1})$ form a Tor-pair in general, even if $R$ is a commutative Gorenstein local ring, thereby supporting the special feature of \Cref{takanswer1} and \Cref{Tor-pair}. In the following, for each $M\in\mo(R)$, denote by $\Omega^n_R(M)$ the $n$-th syzygy of $M$ by choosing a projective resolution in $\mo(R)$; note that Schanuel’s Lemma yields that $\Omega^n_RM$ is independent of the
choice of the projective resolution of $M$ up to projective summands.

\begin{example}\label{counterexample}
  (1) Let $(S,\mathfrak n,k)$ be a regular local ring of positive dimension and $0\neq f \in \mathfrak n^2$. Fix an integer $n\ge 2$ and set  $R=S\llbracket x\rrbracket /(f, x^n)$. Let $d:=\dim R$. For each $M,N\in \mo(R)$, let $\cx_R(M,N)$ denote the complexity of the pair $(M,N)$ as in \cite[page 286]{AB2000}, and set $\cx_R(M)=\cx_R(M,k)$. 
Set $$\X:=\{M\in \MCM (R) \mid  \cx_R(M)\le 1\}.$$ 
Note that $\X$ is a resolving subcategory of $\mo (R)$. Moreover, in $\mo(R)$,
$$
\syz^d_R k\in \MCM (R)\setminus \X={}^{\perp_1}(\X^{\perp_1})\setminus \X \text{ and }\syz^d_R k \in {}^{\top_1}(\X^{\top_1})\setminus \X.
$$

First, we claim  $\X^{\perp_{\infty}}=\X^{\perp_1}=\mathcal{P}^{<\infty}(R)$ in $\mo (R)$.  Indeed, the first equality follows because $\X$ is resolving. For the second equality, $\mathcal{P}^{<\infty}(R)\subseteq \X^{\perp_1}$ is clear since $R$ is Gorenstein.
For the reverse inclusion, let $N\in \X^{\perp_1}=\X^{\perp_{\infty}}$. It remains to show $N\in \mathcal{P}^{<\infty}$. If not, assume $N\notin \mathcal{P}^{<\infty}(R)$. Then, since $R$ is a complete intersection local ring of codimension $2$, $\cx_R(N)=1$ or $2$. If $\cx_R (N)=1$, then $\syz^d_R N\in \X$, and $\cx_R(\syz^d_R N, N)=\cx_R(N,N)=\cx_R(N)=1$ (see \cite[Theorem II]{AB2000}), thus $N\notin (\syz^d_RN)^{\perp_{\infty}}$, contradicting $N\in \X^{\perp_{\infty}} \text{ and } \syz^d_R N\in \X$. Similarly, if $\cx_R(N)=2$, then again by \cite[Theorem II]{AB2000}, $$\cx_R(\syz^d_R R/(x), N)\ge \cx_R(\syz^d_R R/(x))+\cx_R(N)-2=\cx_R(R/(x))=1,$$ 
and hence $N\notin (\syz^d_R R/(x))^{\perp_{\infty}}$ which again contradicts $\syz^d_R R/(x) \in \X$ and $N\in \X^{\perp_{\infty}}$. Thus, we must have $\X^{\perp_1}=\mathcal{P}^{<\infty}(R)$,  and hence $\syz^d_R k\in \MCM (R)\setminus \X={}^{\perp_1}(\X^{\perp_1})\setminus \X.$

\vskip5pt
Next, we claim $\X^{\top_{\infty}}=\X^{\top_1}=\mathcal{P}^{<\infty}(R)$. Indeed, the first equality is because $\X$ is resolving. It is enough to show $\X^{\perp_{\infty}}=\X^{\top_{\infty}}$. Note that $\X=\syz^i_R \X$ for all $i>0$. Indeed, $\syz^i_R \X\subseteq \X$ is clear. For the converse direction, for each $M\in \X$, there exists $N\in \MCM(R)$ such that $M\cong \syz^i_R N$ up to free summand (since $R$ is Gorenstein), and since complexity does not change with syzygy and free summand, we get $N\in \X$. Thus, $\X^{\top_{\infty}}=(\syz^d_R \X)^{\perp_{\infty}}=\X^{\perp_{\infty}}$, where the first equality is by \cite[Theorem III]{AB2000}. Thus, ${}^{\top_1}(\X^{\top_1})={}^{\top_1}(\mathcal{P}^{<\infty}(R))=\MCM(R)$, where the last equality follows from the proof of \Cref{extorthmcm} and \cite[Theorem 5.3.10]{Christensen}. Thus, $\syz^d_R k \in {}^{\top_1}(\X^{\top_1})\setminus \X.$

\vskip5pt
    (2) Let $R$ be a  local Cohen--Macaulay ring of positive dimension $d$ which is not reduced (e.g., $R=k\llbracket x,y\rrbracket /(x^2)$). We claim that in $\mo (R)$, we have $$\res(\syz^d_R k)\neq \MCM(R)={}^{\perp_1}(\res(\syz^d_R k)^{\perp_1}),$$
    where $\res(\syz^d_R k)$ represents the smallest resolving subcategory of $\mo(R)$ containing $\syz^d_Rk$.
    Indeed, since $\res(\syz^d_R k)\subseteq \MCM(R)$, we have $\mathcal{I}^{<\infty}(R)\subseteq \res(\syz^d_R k)^{\perp_1}$ (see \cite[Exercise 3.1.24]{BH}). Combining with \cite[Proposition 3.1.14]{BH}, we get  $\res(\syz^d_R k)^{\perp_1}=\mathcal{I}^{<\infty}(R)$. By \cite[Exercise 3.1.24]{BH} we have $\MCM(R)\subseteq {}^{\perp_1}(\mathcal{I}^{<\infty}(R))$, and then combining with \Cref{extorthmcm} we get $\MCM(R)={}^{\perp_1}(\mathcal{I}^{<\infty}(R))={}^{\perp_1}(\res(\syz^d_R k)^{\perp_1})$. Finally, if we had $\res(\syz^d_R k)=\MCM(R)$, then every maximal Cohen--Macaulay $R$-module would be locally free on punctured spectrum, and $R$ would be an isolated singularity. In particular, $R$ would be reduced (since $\dim R>0$), which leads to a contradiction. 

    \vskip5pt

    Similarly, we also have in $\mo (R)$, $$\res(\syz^d_R k)\neq \MCM(R)={}^{\top_1}(\res(\syz^d_R k)^{\top_1}).$$ The proof is similar by using \cite[Theorem 5.3.10]{Christensen} in place of \cite[Exercise 3.1.24]{BH} and using \cite[Corollary 1.3.2]{BH} in place of \cite[Proposition 3.1.14]{BH}. 
\end{example}

Next, we show that the finitely generated Gorenstein projective modules of an Artin algebra have a more explicit characterization than \Cref{characterization}. 

\vskip5pt
Let $\Lambda$ be an Artin algebra. Choose a projective resolution of $D(\Lambda_\Lambda)$
\[
\xymatrix@R=0.2cm{
\cdots\ar[r]& P_2\ar[rr]^-{\del_2}\ar@{->>}[dr]_{p_2} & & P_1\ar@{->>}[dr]_{p_1}\ar[rr]^-{\del_1}& & P_0\ar[r]^{p_0} &D(\Lambda_\Lambda)\ar[r]&0,\\
&&K_2\ar@{^(->}[ur]_{t_1}&&K_1\ar@{^(->}[ur]_{t_0}&&&
}
\]
where $P_i\in \proj(\Lambda)$ and $K_i$ is the image of $\del_i$ for each $i\geq 1$. Also choose an injective coresolution of $\leftindex_\Lambda\Lambda$
\[
\xymatrix@R=0.2cm{
0\ar[r] & \leftindex_{\Lambda}\Lambda\ar[r]^{q_0} & I_0\ar[rr]^-{\del_0}\ar@{->>}[dr]_{s_0} & & I_{1}\ar[rr]^-{\del_{1}}\ar@{->>}[dr]_{s_1} & & I_{2}\ar[r]^-{\del_{2}}&\cdots, \\
& & & L_{1}\ar@{^(->}[ur]_{q_1} & & L_{2}\ar@{^(->}[ur]_{q_2} & &
}
\]
where $I_i\in \inj(\Lambda)$ is finitely generated injective and $L_i$ is the image of $\del_i$ for each $i\leq 0$.

\vskip5pt

Define $\mathcal S_\Lambda\colonequals \{D(\Lambda_\Lambda), K_1, \cdots, K_n,\cdots\}$ and $\mathcal T_\Lambda\colonequals\{\leftindex_\Lambda\Lambda, L_{1}, \cdots, L_{n}, \cdots\}$. 

\begin{proposition}\label{char-GP-Artin}
    Keep the notations above. Then in $\mo(\Lambda)$ there is
    \[
\Gproj(\Lambda)=\leftindex^{\perp_1}(\mathcal S_\Lambda \cup \mathcal T_\Lambda). 
    \]
\end{proposition}
\begin{proof}
    By \Cref{thick subcategory} (2), $\Gproj(\Lambda)\subseteq \leftindex^{\perp_1}{\Thick(\{\leftindex_\Lambda\Lambda\}\cup \{D(\Lambda_\Lambda)\})}$. Since $\mathcal S_\Lambda \cup \mathcal T_\Lambda\subseteq \Thick(\{\leftindex_\Lambda\Lambda\}\cup \{D(\Lambda_\Lambda)\})$, there is $\leftindex^{\perp_1}{\Thick(\{\leftindex_\Lambda\Lambda\}\cup \{D(\Lambda_\Lambda)\})}\subseteq \leftindex^{\perp_1}(\mathcal S_\Lambda \cup \mathcal T_\Lambda)$. Thus, $\Gproj(\Lambda)\subseteq \leftindex^{\perp_1}(\mathcal S_\Lambda \cup \mathcal T_\Lambda)$. 

    \vskip5pt

    Conversely, let $G\in \leftindex^{\perp_1}{(\mathcal S_\Lambda \cup \mathcal T_\Lambda)}$. Since $G\in \leftindex^{\perp_1}{\ \mathcal T}_\Lambda$, $G\in \leftindex^{\perp_\infty}\Lambda$. According to \Cref{complete resolution}, it suffices to find an exact sequence of $\Lambda$-modules
    \[
0 \rightarrow G\xrightarrow{d_0} P_{-1} \xrightarrow{d_{-1}} P_{-2}\xrightarrow{d_{-2}}\cdots\rightarrow P_{-n}\rightarrow \cdots,
\]
where $P_{-i}\in \proj(\Lambda)$ and the natural embedding $\Ima d_{-(i-1)}\rightarrow P_{-i}$ is a left $\proj(\Lambda)$-approximation for each $i\geq 1$. 

\vskip5pt

Let $d_0: G\rightarrow P_{-1}$ be a left $\proj(\Lambda)$-approximation. We claim that $d_0$ is also a left $(\mathcal S_\Lambda \cup \mathcal T_\Lambda)$-approximation. Let $M\in \mathcal S_\Lambda \cup \mathcal T_\Lambda$. There are three cases: 

\vskip5pt

$(1)$\ $M=\leftindex_\Lambda\Lambda$. Then $\Hom_\Lambda(d_0, \leftindex_\Lambda\Lambda)$ is surjective since $d_0$ is a left $\proj(\Lambda)$-approximation. 

\vskip5pt

$(2)$\ $M=K_i$, $i\geq 0$, $K_0=D(\Lambda_\Lambda)$. For $f: G\rightarrow K_i$, since $\Ext_\Lambda^1(G, K_{i+1})=0$, $f$ factors through $p_i\colon P_i\rightarrow K_i$, say with $f=p_i g$, where $g\colon G\rightarrow P_i$. Then $g$ factors through $d_0$ as $d_0$ is a left $\proj(\Lambda)$-approximation, say $g=hd_0$, where $h\colon P_{-1}\rightarrow P_i$. Thus, $f=p_i g= p_i hd_0$. It follows that $\Hom_\Lambda(d_0, K_i)$, $i\geq 0$ is surjective. 
  \[
\xymatrix{
&&P_{-1}\ar@{..>}[d]_h& G\ar[d]^{f}\ar@{-->}[ld]_g\ar[l]_{d_0}&\\
0\ar[r] & K_{i+1}\ar[r]^{t_i} & P_i\ar[r]^-{p_i} & K_i\ar[r] & 0
}
  \]

$(3)$\ $M=L_i$, $i\geq 1$. Define that $L_0\colonequals\leftindex_\Lambda\Lambda$. For $f: G\rightarrow L_i$, since $\Ext_\Lambda^1(G, L_{i-1})=0$, $f$ factors through $s_{i-1}\colon I_{i-1}\rightarrow L_i$, say with $f=s_{i-1} g$, where $g\colon G\rightarrow I_{i-1}$. Then $g$ factors through $d_0$ as it has been proved above that $\Hom_\Lambda(d_0, D(\Lambda_\Lambda))$ is surjective. Say $g=hd_0$, where $h\colon P_{-1}\rightarrow I_{i-1}$. Thus, $f=s_{i-1} g= s_{i-1} hd_0$. It follows that $\Hom_\Lambda(d_0, K_i)$, $i\geq 0$ is surjective. 
  \[
\xymatrix{
&&P_{-1}\ar@{..>}[d]_h& G\ar[d]^{f}\ar@{-->}[ld]_g\ar[l]_{d_0}&\\
0\ar[r] & L_{i-1}\ar[r]^{q_{i-1}} & I_{i-1}\ar[r]^-{s_{i-1}} & L_i\ar[r] & 0 
}
  \]

Now it has been proved that $d_0$ is a left $(\mathcal S_\Lambda \cup \mathcal T_\Lambda)$-approximation. In particular, $d_0$ is a monomorphism since there is an embedding $\alpha:G\rightarrow D(\Lambda_\Lambda)^n$ and $\alpha$ factors through $d_0$. Consider the short exact sequence
\[
\xymatrix{
    0\ar[r] & G\ar[r]^-{d_0} & P_{-1}\ar[r]^-{u_0} & G_{-1}\ar[r] & 0
    }
\]
For $M\in \mathcal S_\Lambda \cup \mathcal T_\Lambda$, applying $\Hom_\Lambda(-, M)$ to the exact sequence above gives an exact sequence
\[
\xymatrix{
     \Hom_{\Lambda}(P_{-1},M)\ar@{->>}[rr]^{\Hom_\Lambda(d_0, M)}&&\Hom_{\Lambda}(G,M)\ar[r] & \Ext_{\Lambda}^1(G_{-1}, M)\ar[r] & 0
    }
\]
which shows that $\Ext_{\Lambda}^1(G_{-1}, M)=0$. In other words, $G_{-1}\in \leftindex^{\perp_1}{(\mathcal S_\Lambda \cup \mathcal T_\Lambda)}$. Thus, as having been proved above, there is a monic left $(\mathcal S_\Lambda \cup \mathcal T_\Lambda)$-approximation $v_1: G_{-1}\rightarrow P_{-2}$ such that $G_{-2}\colonequals \Coker v_1\in \leftindex^{\perp_1}{(\mathcal S_\Lambda \cup \mathcal T_\Lambda)}$. Then, the desired long exact sequence 
 \[
0 \rightarrow G\xrightarrow{d_0} P_{-1} \xrightarrow{d_{-1}} P_{-2}\xrightarrow{d_{-2}}\cdots\rightarrow P_{-n}\rightarrow \cdots
\]
follows by induction. This completes the proof. 
\end{proof}

Recall that a module $M\in\mo(\Lambda)$ is called \emph{semi-Gorenstein projective} provided that $\Ext^i_\Lambda(M,\Lambda)=0$ for all $i>0$; see \cite{RZ} for more details. Note that a module $M\in\mo(\Lambda)$ is semi-Gorenstein projective if and only if $M\in \leftindex^{\perp_1}{\ \mathcal T}_\Lambda$. Combining this with \Cref{char-GP-Artin}, we get:
\begin{corollary}\label{Gpprop}
 Let $\Lambda$ be an Artin algebra and $M$ be a semi-Gorenstein projective $\Lambda$-module. The following are equivalent. 
 \begin{enumerate}
     \item $M$ is Gorenstein projective.


     \item $M\in \leftindex^{\perp_1}{\mathcal S}_\Lambda$. 
 \end{enumerate}
\end{corollary}




In the following, for $M \in \mo(\Lambda)$, the \emph{Auslander bound}
$\Ab(M,\mo(\Lambda))$, introduced by J.~Wei \cite{wei}, is defined as
\[
\Ab(M,\mo(\Lambda))
= \sup \bigl\{\, e(M,N) \;\big|\; N \in \mo(\Lambda)
\text{ and } e(M,N) < \infty \,\bigr\},
\]
where
$
e(M,N)
= \sup \bigl\{\, n \in \mathbb{N} \;\big|\;
\Ext^n_{\Lambda}(M,N) \neq 0 \,\bigr\}.
$

\begin{corollary}\label{Auslanderbound} Let $\Lambda$ be an Artin algebra. If $M\in \mo(\Lambda)$ is semi-Gorenstein projective and $\Ab(M, \mo(\Lambda))<\infty$, then $M$ is Gorenstein projective.   
\end{corollary}

\begin{proof} 
By \Cref{Gpprop}, it is equivalent to show $M\in \leftindex^{\perp_1}{\mathcal S}_\Lambda$. Namely, we need to show
$\Ext^1_{\Lambda}(M,K_n)=0$ for all $n\ge 1$, where $K_n$ is the $n$-th syzygy in some projective resolution of $D(\Lambda_\Lambda)$. Indeed, since $M$ is semi-Gorenstein projective and $\Ext_{\Lambda}^{>0}(M,D(\Lambda_\Lambda))=0$, we have $\Ext_{\Lambda}^{i}(M,K_1)=0$ for $i>1$. Since $K_{j+1}$ is the first syzygy of $K_j$ for each $j\geq 1$, we get isomorphisms $$\Ext_{\Lambda}^{i+1}(M,K_1)\cong \Ext_{\Lambda}^{i+2}(M,K_2) \cong \cdots \cong \Ext_{\Lambda}^{i+j}(M, K_j)$$ 
for $i\ge 0$ and $j\ge 1$. Combining this with $\Ext^i_\Lambda(M,K_1)=0$ for $i>1$,   we conclude that  $\Ext^{i+j}_{\Lambda}(M, K_j)= 0$ for $i>0$ and $j\geq 1$. If $\Ext^1_\Lambda(M,K_1)\neq 0$, then the above isomorphisms will show that $\Ext^j_\Lambda(M,K_j)\neq 0$ for all $j\geq 1$, and hence $\Ab(M,\mo(\Lambda))\geq j$ for all $j\geq 1$. This contradicts $\Ab(M,\mo(\Lambda))<\infty$. Thus, $\Ext^1_\Lambda(M,K_1)=0$, and hence $\Ext^i_\Lambda(M,K_1)=0$ for all $i>0$.

\vskip5pt
A similar argument as above, with $D(\Lambda_\Lambda)$ replaced by $K_1$ and $K_1$ replaced by $K_2$, shows that
$
\Ext^{>0}_{\Lambda}(M, K_2)=0.
$
Iterating this argument yields the same conclusion for every $K_n$, as desired.
\end{proof}

\begin{remark}

 (1) It follows from the proof of \Cref{Auslanderbound} that if $\Lambda$ is an Artin algebra, and $M\in \mo(\Lambda)$ is semi-Gorenstein projective such that $\Ab(M, {\mathcal S}_\Lambda)<\infty$, then $M$ is Gorenstein projective.

\vskip5pt
(2) Let $M\in \mo(\Lambda)$ have finite reducing Gorenstein dimension as in \cite[Definition 2.5]{reducing}. Then there exist short exact sequences $$\{0\to M_{i-1}^{\oplus a_i}\to M_i \to \Omega^{n_i}_\Lambda M_{i-1}^{\oplus b_i} \to 0\}_{i=1}^r$$ in $\mo(\Lambda)$, where $n_i\ge 0$, $a_i,b_i>0$, $M_0=M$, and $M_r$ has finite Gorenstein projective dimension, say $d$.  Taking $d$-th syzygies of these exact sequences, and remembering $\Omega^d_\Lambda M_r$  is Gorenstein projective, and since $\mathcal S_{\Lambda} \in \Gproj(\Lambda)^{\perp_{\infty}}$, a same proof as in  \cite[5.1]{redcelik} shows $\Ab(\Omega^d_\Lambda M, \mathcal S_\Lambda)<\infty$, i.e., $\Ab( M, \mathcal S_\Lambda)<\infty$. Thus, if $\Ext_{\Lambda}^{\gg 0}(M, \Lambda)=0$, then passing to high enough syzygy of $M$, and remembering $\Ab( \Omega^n_\Lambda M, \mathcal S_\Lambda)<\infty$, for all $n\ge 0$, we get by (1) that $M$ has finite Gorenstein projective dimension. This gives an alternative proof of \cite[Corollary 2.9]{reducing} for Artin algebras.

\end{remark}
\begin{remark}
    Let $S$ be a Cohen--Macaulay local ring, and let $R$ be a noetherian
$S$-algebra such that $R \in \MCM(S)$. For $M \in \mo(R)$, if
$M \in \MCM(S)$, $M$ is semi--Gorenstein projective over $R$, and
$\Ab(M,\mo(R)) < \infty$, then $M$ is Gorenstein projective over $R$.

\vskip5pt
Indeed, choose a maximal $S$-regular sequence
$\x = x_1,\ldots,x_d$. Then $\x$ is both $R$-regular and
$M$-regular; see \cite[Corollary~1.6.19]{BH}. It can be verified directly using \cite[Lemma 2 (ii) in Chapter 6, Section 18]{matsumura} that the Artin algebra $R/\x R$ and the module
$M/\x M$ satisfy the assumption of \Cref{Auslanderbound}. Hence, $M/\x M$ is Gorenstein
projective over $R/\x R$. Finally, \Cref{reduction} (1) shows that $M$ is Gorenstein
projective over $R$.
\end{remark}

\section{Weakly Gorensteinness and virtually Gorensteinness}\label{WGVG}

We prove \Cref{T1} in this section; see \Cref{char-weakly Gor}. 
It provides a characterization of left weakly Gorenstein rings for certain noetherian algebras over Cohen--Macaulay rings.
We also study virtually Gorenstein rings and obtain an equivalent characterization related to a question of X.-W. Chen \cite{Chen-postdoc}; see \Cref{G-free} and \Cref{char-VG}. 

\vskip5pt

\begin{chunk}\label{def of weakly Gor}
     Following C. M. Ringel and P. Zhang \cite[1.1]{RZ}, a noetherian ring $R$ is said to be \emph{left weakly Gorenstein} if $\leftindex^{\perp_\infty} R=\Gproj(R)$ in $\mo (R)$. We will say that $R$ is \emph{weakly Gorenstein} if both $R$ and $R^{\rm op}$ are left weakly Gorenstein.

\vskip5pt
     
   Note that $\leftindex^{\perp_\infty} R=\leftindex^{\perp_\infty} ({\mathcal P}^{<\infty}(R))$ in $\mo(R)$. By \Cref{def-cotorsion}, any Iwanaga--Gorenstein ring is weakly Gorenstein. Also, \Cref{characterization} yields that any commutative Gorenstein ring is weakly Gorenstein.
 \end{chunk}
   The following example is a weakly Gorenstein Cohen--Macaulay local ring but not Gorenstein.
\begin{example}
    Let $(R,\m)$ be a commutative local Artin ring with $\m^2=0$, and assume that $R$ is not Gorenstein. For example, $R=k\llbracket x,y\rrbracket/(x,y)^2$, where $k$ is a field. Then $R$ is a weakly Gorenstein Cohen--Macaulay local ring. Indeed, there is $\Gproj(R)=\proj(R)=\leftindex^{\perp_\infty} R$; see \cite[Proposition 2.4]{Yoshino2003}.
\end{example}
The following proposition is a characterization of left weakly Gorenstein rings. The equivalences in \Cref{recover} are essentially proved by C. Huang and Z. Y. Huang \cite[Claim 5.1]{HH}. When $R$ is an Artin algebra, the equivalences in \Cref{recover} are also established by C. M. Ringel and P. Zhang in \cite[Theorem 1.2]{RZ}. 

\begin{proposition}\label{recover}
Let $R$ be a noetherian ring. The following are equivalent.
    \begin{enumerate}
        \item $R$ is left weakly Gorenstein. 

        \item Any module in $\leftindex^{\perp_\infty} R$ is reflexive.

        \item Any module in $\leftindex^{\perp_\infty} R$ is torsionless. 
    \end{enumerate}
\end{proposition}
\begin{proof}

    $(1)\Rightarrow (2)$. If $R$ is left weakly Gorenstein, then any module in $\leftindex^{\perp_\infty} R$ is Gorenstein projective, and hence any module in $\leftindex^{\perp_\infty} R$ is reflexive.

\vskip5pt

    The implication $(2)\Rightarrow (3)$ is trivial. 

\vskip5pt

    $(3)\Rightarrow (1)$. Assume any module in $ 
 \leftindex^{\perp_\infty} R$ is torsionless. We use $(-)^\ast$ to denote the functor $\Hom_R(-,R)$ or $\Hom_{R^{\rm op}}(-,R)$, depending on the context.
 For each $M\in \leftindex^{\perp_\infty} R$, choose a short exact sequence in $\mo(R)$
    $$
    0\rightarrow K\rightarrow P\xrightarrow\pi M^\ast\rightarrow 0,
    $$
    where $P\in \proj(R)$. Applying $(-)^\ast$ to the above short sequence, we get a short exact sequence
    $$
    0\rightarrow M^{\ast\ast}\xrightarrow{\pi^\ast} P^\ast\rightarrow C\rightarrow 0. 
    $$
    Since the canonical map $\varphi_M\colon M\rightarrow M^{\ast\ast}$ is injective, we have a short exact sequence
    $$
    0\rightarrow M\xrightarrow{\pi^\ast\varphi_M} P^\ast\rightarrow C^\prime\rightarrow 0.
    $$
    Combining this with $\pi$ is surjective, we can get that 
    $(\pi^\ast \varphi_M)^\ast\colon P^{\ast\ast}\rightarrow M^\ast$ is surjective. Indeed, this follows from the following commutative diagram
    $$
    \xymatrix{
P \ar[r]^-\pi \ar[d]_-{\varphi_P}& M^\ast\ar[d]_-{\varphi_{M^{\ast}}}\ar@{=}[rd]& \\
P^{\ast\ast}\ar[r]_-{\pi^{\ast\ast}}& M^{\ast\ast\ast}\ar[r]_-{(\varphi_M)^\ast}& M^\ast
    }.
    $$
Thus, we get that $\pi^\ast \varphi_M$ is a monic left  $\proj(R)$-approximation of $M$. 

\vskip5pt

 Since $\pi^\ast\varphi_M$ is a monic left $\proj(R)$-approximation and $M\in \leftindex^{\perp_\infty} R$, we get that $C^\prime \in\leftindex^{\perp_\infty} R$.
Similarly, one obtains a monic left $\proj(R)$-approximation of $C^\prime$. 
  By iterating this process, we can construct a projective coresolution of $M$
satisfying the conditions of \Cref{complete resolution}, and hence $M\in \Gproj(R)$. Consequently, $\Gproj(R)=\leftindex^{\perp_\infty} R$. That is, $R$ is left weakly Gorenstein. 
\end{proof}

\vskip5pt

The functor $(-)^\dagger=\Hom_S(-,\omega)$ is defined as the setting of the previous section.
\begin{lemma}\label{cm case-approximation}
   Let $S$ be a Cohen--Macaulay ring that admits a canonical module $\omega$. Assume that $R$ is a noetherian $S$-algebra and that $R$, viewed as an $S$-module, is maximal Cohen--Macaulay. Fix a short exact sequence
    \[
    \xymatrix{
    0\ar[r] & X\ar[r] & P\ar[r]^-{\pi} & R^\dagger\ar[r] & 0
    }
    \]
    where $P\in \proj(R)$. If $\Ext^i_R(M, X\oplus R)=0$ for all $i>0$, then $M$ has a monic left $\proj(R)$-approximation.  
\end{lemma}
\begin{proof}
  Let $f\colon M\rightarrow Q$ be a left $\proj(R)$-approximation for some $Q\in\proj(R)$. Since $M\in \leftindex^{\perp_\infty} (X\oplus R)$, we conclude that $M\in \leftindex^{\perp_\infty} (R^\dagger)$ by the above short exact sequence.  By \Cref{global version} (1) and \Cref{change of ring-iso}, $M\in\MCM(S)$ as an $S$-module. By \Cref{monomorphism}, we can choose an embedding $h\colon M \rightarrow (R^\dagger)^n$. Since $\Ext_R^1(M, X^n)=0$, $h$ factors through $\pi^n\colon P^n\rightarrow (R^\dagger)^n$, say with $h=\pi^n g$, where $g\colon M\rightarrow P^n$. Then $g$ factors through $f$ as $f$ is a left $\proj(R)$-approximation, say $g=tf$, where $t\colon Q\rightarrow P^n$. Thus, $h=\pi^n g= \pi^n tf$, and hence $f$ is a monomorphism since $h$ is. 
  \[
\xymatrix{
&&Q\ar@{..>}[d]_t& M\ar[d]^{h}\ar@{-->}[ld]_g\ar[l]_{f}&\\
0\ar[r] & X^n\ar[r] & P^n\ar[r]^-{\pi^n} & (R^\dagger)^n\ar[r] & 0
}
  \]
\end{proof}


\begin{theorem}\label{char-weakly Gor}
   Let $S$ be a Cohen--Macaulay ring that admits a canonical module $\omega$. Assume that $R$ is a noetherian $S$-algebra and that $R$, viewed as an $S$-module, is maximal Cohen--Macaulay. The following are equivalent.
\begin{enumerate}
    \item $R$ is left weakly Gorenstein.

    \item $\Thick(\proj(R)\cup \{R^\dagger\})\subseteq (\leftindex^{\perp_\infty} {R})^{\perp_\infty}$. 

    \item $\Thick(\proj(R)\cup \{R^\dagger\})\subseteq (^{\perp_\infty} R)^{\perp_1}$. 

    \item $X\in (\leftindex^{\perp_\infty} R)^{\perp_\infty}$, where $X$ is as in \Cref{cm case-approximation}. 

    \item $X\in (\leftindex^{\perp_\infty} R)^{\perp_1}$, where $X$ is as in \Cref{cm case-approximation}. 
\end{enumerate}
\end{theorem}

\begin{proof}
$(1)\Rightarrow (2)$. Assume that $R$ is left weakly Gorenstein. Then $\Gproj(R)=\leftindex^{\perp_\infty} R$. Hence, the statement of $(2)$ is equivalent to that  $\Thick(\proj(R)\cup \{R^\dagger\})\subseteq \Gproj(R)^{\perp_\infty}$. This follows immediately from \Cref{infty setting}.

\vskip5pt

The implication $(2)\Rightarrow (3)$ is trivial. 

\vskip5pt

$(3)\Rightarrow (1)$. The condition $\Thick(\proj(R)\cup \{R^\dagger\})\subseteq (^{\perp_\infty} R)^{\perp_1}$ is equivalent to say that $^{\perp_\infty} R\subseteq \leftindex^{\perp_1}{\Thick(\proj(R)\cup \{R^\dagger\})}$. It follows that $\Gproj(R)\subseteq \leftindex^{\perp_\infty}{R}\subseteq \leftindex^{\perp_1}{\Thick(\proj(R)\cup \{R^\dagger\})}=\Gproj(R)$. 

\vskip5pt

The implication $(2)\Rightarrow (4)$ is trivial as $X\in \Thick(\proj(R)\cup\{R^\dagger\})$. 

\vskip5pt

$(4)\Rightarrow (1)$. By \Cref{cm case-approximation} and assumption, any $M\in \leftindex^{\perp_\infty} R$ has a monic left $\proj(R)$-approximation. In particular, any $M\in \leftindex^{\perp_\infty} R$ is torsionless; see \Cref{torsionless}. By \Cref{recover}, $R$ is left weakly Gorenstein.  

\vskip5pt

The implication $(4)\Rightarrow (5)$ is trivial. 

\vskip5pt

$(5)\Rightarrow (4)$. It suffices to prove that for $M\in \leftindex^{\perp_\infty} R$ and $n\geq 2$, there is $\Ext_R^n(M, X)=0$. Choose a projective resolution of $M$
\[
\xymatrix@R=0.2cm{
\cdots\ar[r]& P_2\ar[rr]^-{\del_2}\ar@{->>}[dr] & & P_1\ar@{->>}[dr]\ar[rr]^-{\del_1}& & P_0\ar[r] &M\ar[r]&0.\\
&&K_2\ar@{^(->}[ur]&&K_1\ar@{^(->}[ur]&&&
}
\]
Note that each $K_i\in \leftindex^{\perp_\infty} R$, $i\geq 1$. It follows that for $n\geq 2$, $\Ext_R^n(M, X)=\Ext_R^1(K_{n-1}, X)=0$. This completes the proof. 
\end{proof}

\begin{corollary}
    Keep the assumptions as in \Cref{char-weakly Gor}. If, in addition, $R^\dagger\in \mathcal P^{<\infty}(R)$, then $R$ is left weakly Gorenstein. 
\end{corollary}
\begin{proof}
    Let $X$ be as in \Cref{cm case-approximation}. Since $ R^\dagger\in \mathcal P^{<\infty}(R)$, $X\in \mathcal P^{<\infty}(R)$. Thus, $X\in \mathcal P^{<\infty}(R)\subseteq (\leftindex^{\perp_\infty} R)^{\perp_1}$. By \Cref{char-weakly Gor}, $R$ is left weakly Gorenstein. 
\end{proof}

As an application of the above result, we have the following three consequences.  

\begin{corollary}\label{finite Krull dim case}
       Keep the assumptions as in \Cref{char-weakly Gor}. If, in addition, $\dim(S)<\infty$, then the following are equivalent. 

       \begin{enumerate}
    \item $R$ is left weakly Gorenstein.

    \item $\Thick(\proj(R)\cup {\mathcal I}^{<\infty}(R))\subseteq (\leftindex^{\perp_\infty} {R})^{\perp_\infty}$. 
\end{enumerate}
\end{corollary}
\begin{proof}
    $(1)\Rightarrow (2)$. By \Cref{thick subcategory} and \Cref{fid contains}, $\Thick(\proj(R)\cup{\mathcal I}^{<\infty}(R))\subseteq \Gproj(R)^{\perp_\infty}$. The desired result now follows as $R$ is left weakly Gorenstein.

\vskip5pt

    $(2)\Rightarrow (1)$. Since $\dim(S)<\infty$, the canonical module $\omega$ of $S$ has finite injective dimension; see \Cref{def of canonical}. It follows from \Cref{change of ring-iso} that $R^\dagger$ has finite injective dimension over $R$. In particular, $R^\dagger \in {\mathcal I}^{<\infty}(R)$. Combining this with the statement of $(2)$, we have $\Thick(\proj(R)\cup \{R^\dagger\})\subseteq (\leftindex^{\perp_\infty} R)^{\perp_\infty}$. By \Cref{char-weakly Gor}, $R$ is left weakly Gorenstein.
\end{proof}
    

\begin{corollary}
 Keep the assumptions as in \Cref{char-weakly Gor}.  Then $R$ is left weakly Gorenstein if and only if $$\Gproj(R)^{\perp_\infty}=(\leftindex^{\perp_\infty} {R})^{\perp_\infty}.$$ 
\end{corollary}
\begin{proof}
    If $R$ is left weakly Gorenstein, then $\Gproj(R)= \leftindex^{\perp_\infty} R$. This yields that $\Gproj(R)^{\perp_\infty}=(\leftindex^{\perp_\infty} {R})^{\perp_\infty}$. Conversely, assume $\Gproj(R)^{\perp_\infty}=(\leftindex^{\perp_\infty} {R})^{\perp_\infty}$. Note that $X\in \Thick(\proj(R)\cup \{R^\dagger\})$. Combining this with \Cref{infty setting}, we get  $X\in \Gproj(R)^{\perp_\infty}=(\leftindex^{\perp_\infty} {R})^{\perp_\infty}$. By \Cref{char-weakly Gor}, $R$ is left weakly Gorenstein.
\end{proof}


\begin{chunk}\label{question of Chen}
Let $\Lambda$ be an Artin algebra such that $\Gproj(\Lambda) = \proj(\Lambda)$. 
In \cite[Problem C]{Chen-postdoc}, X.-W. Chen asked the following question: does it follow that $\GProj(\Lambda) = \Proj(\Lambda)$? 
\Cref{char-VG} will yield that this question is equivalent to asking whether $\Lambda$ is virtually Gorenstein.

\end{chunk}

\begin{proposition}\label{G-free}
    Let $S$ be a Cohen--Macaulay ring that admits a canonical module $\omega$. Assume that $R$ is a noetherian $S$-algebra and that $R$, viewed as an $S$-module, is maximal Cohen--Macaulay. Then $\GProj(R)=\Proj(R)$ if and only if $\Thick(\proj(R)\cup \{R^\dagger\})=\mo(R)$, where $(-)^\dagger=\Hom_S(-, \omega)$. 

\end{proposition}

\begin{proof}
    First, we recall that $R^\dagger$ is the dualizing complex of the pair $\langle R, R\rangle$ in the sense of  \cite[Definition 3.3.1]{IK2006}. Note that $\GProj(R)=\Proj(R)$ if and only if $\K_{\rm tac}(\Proj(R))=0$, where $\K_{\rm tac}(\Proj(R))$ is the full subcategory consisting of totally acyclic complexes of projective modules in the homotopy category of complexes of $R$-modules. By \cite[Theorems 5.3 (1) and 5.12]{IK2006}, $\K_{\rm tac}(\Proj(R))$ is compactly generated. Thus, $\K_{\rm tac}(\Proj(R))=0$ if and only if $\K_{\rm tac}^{\rm c}(\Proj(R))=0$,  where $\K_{\rm tac}^{\rm c}(\Proj(R))$ is the full subcategory of compact objects in $\K_{\rm tac}(\Proj(R))$. Combining \cite[Theorems 5.3 (2) and 5.12]{IK2006} with \cite[Theorem 1]{KS}, we conclude that $\K_{\rm tac}^{\rm c}(\Proj(R))=0$ if and only if $\Thick(\proj(R)\cup \{R^\dagger\})=\mo(R)$. This completes the proof. 
\end{proof}

\begin{chunk}\label{stable cat}
Let $\underline{\GProj}(\Lambda)$ denote the stable category of $\GProj(\Lambda)$ modulo projective modules. The objects of $\underline{\GProj}(\Lambda)$ are the same as those of $\GProj(\Lambda)$. For any $M, N \in \underline{\GProj}(\Lambda)$, the morphism space
 $$
 \Hom_{\underline{\GProj}(\Lambda)}(M,N)\colonequals\Hom_\Lambda(M,N)/{\mathcal P}(M,N),
 $$
 where $\mathcal{P}(M,N)$ consists of morphisms in $\Hom_\Lambda(M,N)$ which factor through a projective module. 
It is known that $\GProj(\Lambda)$ is a Frobenius category whose projective--injective objects are precisely the projective modules, and hence its stable category $\underline{\GProj}(\Lambda)$ is triangulated. We refer the reader to \cite{Happel} for further details on Frobenius categories and the triangulated structure of their stable categories.

\vskip5pt
When $\Lambda$ is an Artin algebra, A. Beligiannis \cite[Theorem~6.6]{Beligiannis-JA} observed that $\underline{\GProj}(\Lambda)$ is a compactly generated triangulated category; this result also holds for more general rings that admit dualizing complexes; see \cite[Theorems 5.3 and 5.12]{IK2006}.

\end{chunk}

\begin{theorem}\label{char-VG}
    Let $\Lambda$ be an Artin algebra and assume $\Gproj(\Lambda)=\proj(\Lambda)$. The following are equivalent. 
    \begin{enumerate}
        \item $\Lambda$ is virtually Gorenstein.

        \item $\GProj(\Lambda)=\Proj(\Lambda)$.

        \item $\Thick(\proj(\Lambda)\cup \inj(\Lambda))=\mo(\Lambda)$. 

        \item $
    (\leftindex^{\perp_1}\Thick(\proj(\Lambda)\cup \inj(\Lambda)), \Thick(\proj(\Lambda)\cup \inj(\Lambda)))
    $
    is a cotorsion pair.
    \end{enumerate} 
\end{theorem}
\begin{proof}
   $(1)\iff (2)$. By \cite[Theorem 8.2]{Beligiannis-JA}, $\Lambda$ is virtually Gorenstein if and only if $\underline{\GProj}(\Lambda)^c=\underline{\Gproj}(\Lambda)$, where $\underline{\GProj}(\Lambda)^c$ is the full subcategory of compact objects in $\underline{\GProj}(\Lambda)$. Combining this with $\Gproj(\Lambda)=\proj(\Lambda)$, we get that $\Lambda$ is virtually Gorenstein if and only if $\underline{\GProj}(\Lambda)=0$. The latter is equivalent to $\GProj(\Lambda)=\Proj(\Lambda)$. 

   \vskip5pt
   
  $(2)\iff (3)$. This is just \Cref{G-free}; see \Cref{duality}.

  \vskip5pt

  $(3)\iff (4)$. By \Cref{characterization}, \Cref{duality} and assumption, there is $\leftindex^{\perp_1}\Thick(\proj(\Lambda)\cup \inj(\Lambda))=\proj(\Lambda)$. This completes the proof. 
\end{proof}

\section{Gorensteinness via cotorsion pairs}
 If $R$ is an Iwanaga--Gorenstein ring, then it is known that $(\Gproj(R), {\mathcal P}^{<\infty}(R))$ forms a cotorsion pair in $\mo (R)$; see \Cref{def-cotorsion}. The main result of this section shows that the converse holds for Cohen--Macaulay local rings; see \Cref{result3}. This corresponds to \Cref{T3} in the introduction. Before stating the result, we need some preparations.

\begin{lemma}\label{idpd}
    Let $R$ be a commutative noetherian ring.  Assume that there is a finitely generated $R$-module $X$ with finite injective dimension, and $X_\m\neq 0$ for each maximal ideal $\m$ of $R$. If $\Gproj(R)^{\perp_1}={\mathcal P}^{<\infty}(R)$, then $R$ is Gorenstein.
\end{lemma}
\begin{proof}
By \Cref{fid contains}, $X\in \Gproj(R)^{\perp_1}$. Combining this with $\Gproj(R)^{\perp_1}={\mathcal P}^{<\infty}(R)$, we get that $X\in {\mathcal P}^{<\infty}(R)$. Thus, $X$ is a finitely generated $R$-module with both finite projective dimension and finite injective dimension. For each maximal ideal $\m$ of $R$, $X_\m\neq 0$ by assumption. We conclude from this that $X_\m$ has finite projective dimension and finite injective dimension over $R_\m$. It follows from a result of Foxby \cite[Corollary 4.4]{Foxby} that $R_\m$ is Gorenstein. For each prime ideal $\p$ of $R$, $\p$ is contained in a maximal ideal $\m$ of $R$. Note that $\p R_\m$ is a prime ideal of $R_\m$ and $R_\p\cong (R_\m)_{(\p R_\m)}$. Thus, $R_\p$ is Gorenstein as $R_\m$ is. We conclude that $R$ is Gorenstein. This completes the proof.
 \end{proof}




\begin{corollary}\label{dualizingcase}
     Let $R$ be a  Cohen--Macaulay ring with a canonical module $\omega$ and $\dim(R)<\infty$. If $\Gproj(R)^{\perp_1}={\mathcal P}^{<\infty}(R)$, then $R$ is Gorenstein. 
\end{corollary}
\begin{proof}
    Since $\dim(R)<\infty$, there is $\id_R(\omega)<\infty$; see \Cref{def of canonical}. By definition, $\omega_\p$ is a canonical module over $R_\p$ for each prime ideal $\p$ of $R$, and hence $\omega_\p\neq 0$. By \Cref{idpd}, $R$ is Gorenstein. 
\end{proof}

\begin{corollary}\label{CMcase}
    Let $R$ be a Cohen--Macaulay local ring. If  $\Gproj(R)^{\perp_1}={\mathcal P}^{<\infty}(R)$, then $R$ is Gorenstein. 
\end{corollary}
\begin{proof}
Any Cohen--Macaulay local ring admits a nonzero finitely generated module of finite injective dimension; see the discussion following \cite[Corollary 9.6.2]{BH}. The desired result now follows from \Cref{idpd}.
\end{proof}

\begin{theorem}\label{result3}
    Let $R$ be a Cohen--Macaulay local ring. The following are equivalent. 
    \begin{enumerate}
        \item $R$ is Gorenstein.

        \item $\Gproj(R)^{\perp_1}={\mathcal P}^{<\infty}(R)$.

        \item $(\Gproj(R), {\mathcal P}^{<\infty}(R))$ is a cotorsion pair.
    \end{enumerate}
\end{theorem}
\begin{proof}  
$(1)\Rightarrow (3)$. Since $R$ is Gorenstein and $\dim(R)<\infty$, we have $\id_R(R)<\infty$. Then  $(\Gproj(R),{\mathcal P}^{<\infty}(R))$ is a cotorsion pair; see \Cref{def-cotorsion}.

\vskip5pt

The implication $(3)\Rightarrow (2)$ is trivial, and the implication $(2)\Rightarrow (1)$  follows from \Cref{CMcase}.
\end{proof}

\bibliographystyle{amsplain}
\bibliography{ref}
\end{document}